\documentclass[11pt]{amsart}

\usepackage{amsfonts, amstext, amsmath, amsthm, amscd, amssymb}
\usepackage{graphicx, color, epsfig, subfigure, wrapfig, overpic}

\let\oldmarginpar\marginpar
\renewcommand\marginpar[1]{\oldmarginpar[\raggedleft\footnotesize #1]%
{\raggedright\footnotesize #1}}

\newcommand{\Z}{\mathbb{Z}}

\newcommand{\e}{{\varepsilon}}
\newcommand{\vol}{{\rm Vol}}

\newcommand{\bdy}{{\partial}} 

\renewcommand{\setminus}{{\smallsetminus}}

\theoremstyle{plain}
\newtheorem{theorem}{Theorem}[section]
\newtheorem{corollary}[theorem]{Corollary}
\newtheorem{lemma}[theorem]{Lemma}

\newtheorem{claim}[theorem]{Claim}

\theoremstyle{definition}
\newtheorem{define}[theorem]{Definition}

\newtheorem*{namedtheorem}{\theoremname}
\newcommand{\theoremname}{testing}

\title[Volume bounds for generalized twisted torus links]{Volume bounds for generalized twisted torus links}

\author[A.\ Champanerkar]{Abhijit Champanerkar}
\author[D.\ Futer]{David Futer}
\author[I.\ Kofman]{Ilya Kofman}
\author[W.\ Neumann]{Walter Neumann}
\author[J.\ Purcell]{Jessica S. Purcell}

\thanks{\today}

\begin{document}

 \begin{abstract}
 Twisted torus knots and links are given by twisting adjacent strands of a
torus link.  They are geometrically simple and contain many examples of
the smallest volume hyperbolic knots.  Many are also Lorenz links.

We study the geometry of twisted torus links and related generalizations.
We determine upper bounds on their hyperbolic volumes that depend only on
the number of strands being twisted.  We exhibit a family of twisted torus
knots for which this upper bound is sharp, and another family with volumes
approaching infinity.  Consequently, we show there exist twisted torus
knots with arbitrarily large braid index and yet bounded volume.
 \end{abstract}

\maketitle

\section{Introduction}

Recently, there has been interest in relating the volume of a
hyperbolic knot and link to other link properties.  Lackenby has related the 
volume of an alternating link
to the number of twist regions in its diagram
\cite{lackenby:volume-alt}, and this relationship was extended to larger classes of links 
that satisfy a certain threshold of complexity, 
such as a high amount of symmetry or twisting \cite{fkp:dfvjp,
  fkp:symm, purcell:hyp-genaug}.  To better understand volumes in
general, it seems natural to also investigate properties of knots and links
that are ``simple.'' 

Twisted torus knots and links are obtained by twisting a subset of
strands of a closed torus braid.  These knots are geometrically simple
by several different measures of geometric complexity.  Dean
\cite{dean:sff} showed that they often admit small Seifert fibered and
lens space fillings.  In \cite{cdw:simplest, ckp:next-simplest}, it
was discovered that twisted torus knots dominate the census of
``simplest hyperbolic knots,'' those whose complements can be
triangulated with seven or fewer ideal tetrahedra.  It is not
surprising then that twisted torus knots contain many examples of the
smallest volume hyperbolic knots.

In this paper we investigate the geometry of twisted torus links and
closely related generalizations.  We determine upper bounds on their
volumes in terms of their description parameters.  We also exhibit a
family of twisted torus knots for which this upper bound is sharp, and
another family with volumes approaching infinity.

A consequence of these results is that the braid index of a knot or
link gives no indication of its volume.  Using techniques of
\cite{birman-kofman} to determine braid index, we show there exist
twisted torus knots with arbitrarily large braid index and yet bounded
volume.  The reverse result is also known, for example closed
3--braids can have unbounded volume \cite{fkp:farey}.

\subsection{Twisted torus links}

A {\em positive root} $\beta$ will mean a positive $n$--braid whose
$n$--th power is the central element $\Delta_n^2$ (i.e. the full
twist) in $B_n$.  We show in Theorem \ref{thm:roots}
that there are $2^{n-2}$ braid isotopy classes of positive $n$-th
roots, all of which have the form
\begin{equation}\label{eq:roots}
\beta=\sigma_{i_1}\cdots\sigma_{i_{n-1}}
\end{equation} 
with $i_1,\dots,i_{n-1}$ a permutation of $1,\dots,n-1$, and all of
which are conjugate in $B_n$ to $\delta_n=\sigma_1\cdots
\sigma_{n-1}$, which we will call the {\em standard root}. Let 
$\bar\delta_n= \sigma_{n-1}\cdots\sigma_{1}$ denote a conjugate root which 
we will also use below.

\begin{define}\label{def:ttk}
Let $p>0$, $q,s \neq 0$, and $1<r\leq p+|q|$ be integers.  Let $\beta\in
B_r$ be any positive root.  Let $L$ be a $(p,q)$--torus link embedded
on a flat torus.  Let $D$ be a regular neighborhood of a line segment
that crosses $r$ strands of $L$, as in Figure \ref{fig:ttkfig}.  The
\emph{twisted torus link} $T(p,q,r,s,\beta)$ is formed from $L$ by
replacing the $r$ strands of $L\cap D$ with the braid $\beta^s$.  We
will usually suppress the root $\beta$ in the notation.
\end{define}

\begin{figure}
\includegraphics[height=0.74in]{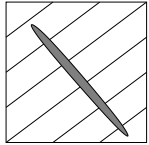}
\caption{The regular neighborhood of a line segment crossing $r$ strands.}
\label{fig:ttkfig}
\end{figure}

In \cite{dean:sff}, Dean defined the twisted torus link $T(p,q,r,s)$
by replacing the $r$ strands of $L\cap D$ with the $r$--braid
$(\delta_r)^s$, for $s$ an integer multiple of $r$, which implies
$T(p,q,r,s)=T(q,p,r,s)$.
In \cite{cdw:simplest, ckp:next-simplest}, twisted torus links were
defined as closed braids for $1<r<p$ and for $s$ an integer multiple
of $r$; namely, the closure of the $p$--braid
$(\delta_p)^q\,(\delta_r)^s=(\sigma_1\cdots\sigma_{p-1})^q
(\sigma_1\cdots \sigma_{r-1})^s$.  In \cite{birman-kofman}, this was
generalized to any integer $s$, for which switching $p$ and $q$ may
result in distinct links.
Definition \ref{def:ttk} includes all of these as special cases.

\subsection{$T$--links}

A natural way to generalize twisted torus links is to repeatedly twist
nested subsets of strands.  

\begin{define}
  Let $ r_1 > \dots > r_k \geq 2 $ and $q,s_i\neq 0$ be integers.
  Define the \emph{ $T$--link} $T((p,q),(r_1, s_1, \beta_1), \dots,
  (r_k,s_k,\beta_k))$ 
to be formed from $L$, as in Definition
  \ref{def:ttk} with $r=r_1$, by replacing the $r$ strands of $L\cap
  D$ with the braid $\beta_{1}^{s_1}\cdots \beta_{k}^{s_k}$, where
  each $\beta_i$ is a specified $r_i$-th root of the full twist on
 the first $r_i$ strands.  We will usually suppress the roots $\beta_i$ in the notation.
If $k=1$, a $T$--link is a twisted torus
  link. See Figure \ref{fig:T-knots} for an example. 
\end{define}

\begin{figure}[h]
  \begin{center}
    \subfigure[{\tiny $T(9,7)$}]{\includegraphics[scale=0.37]{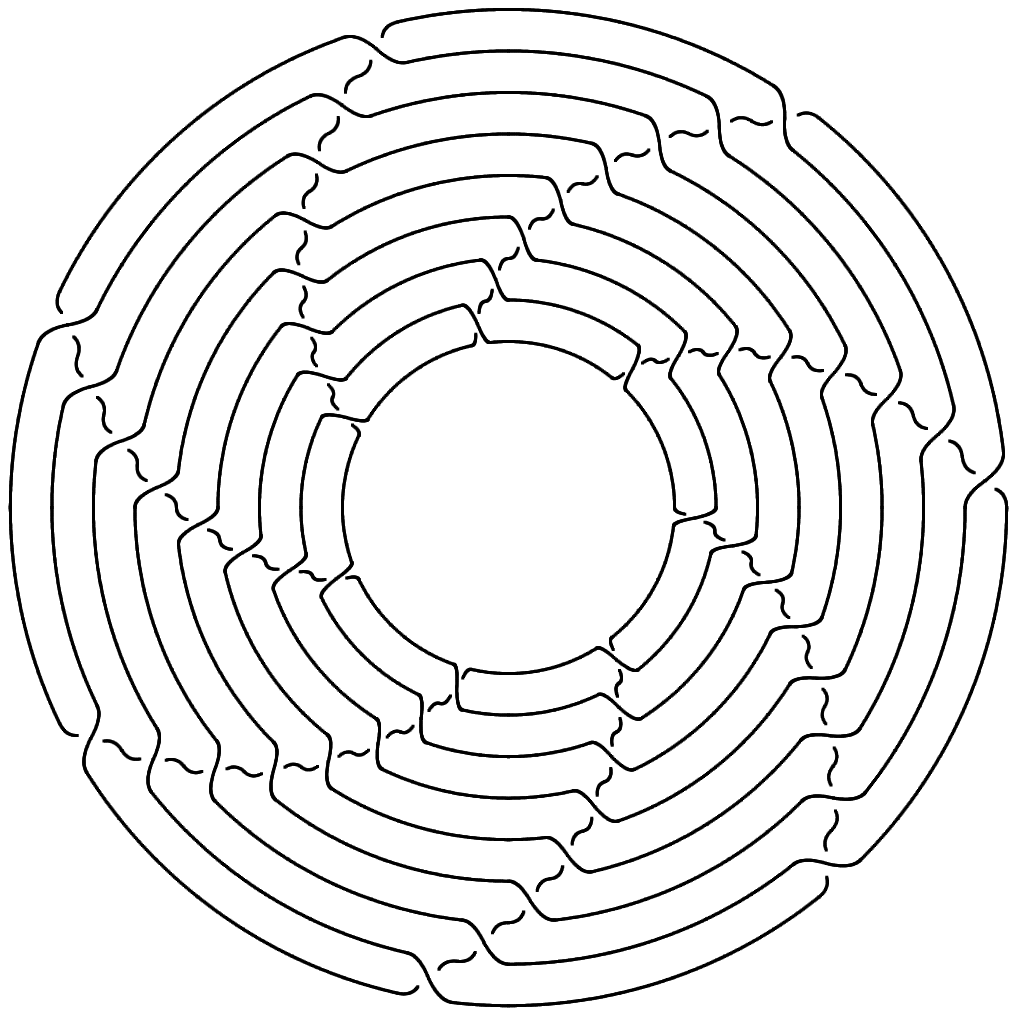}}
\qquad\quad
\subfigure[{\tiny $T(9,7,5,3)$}]{\includegraphics[scale=0.37]{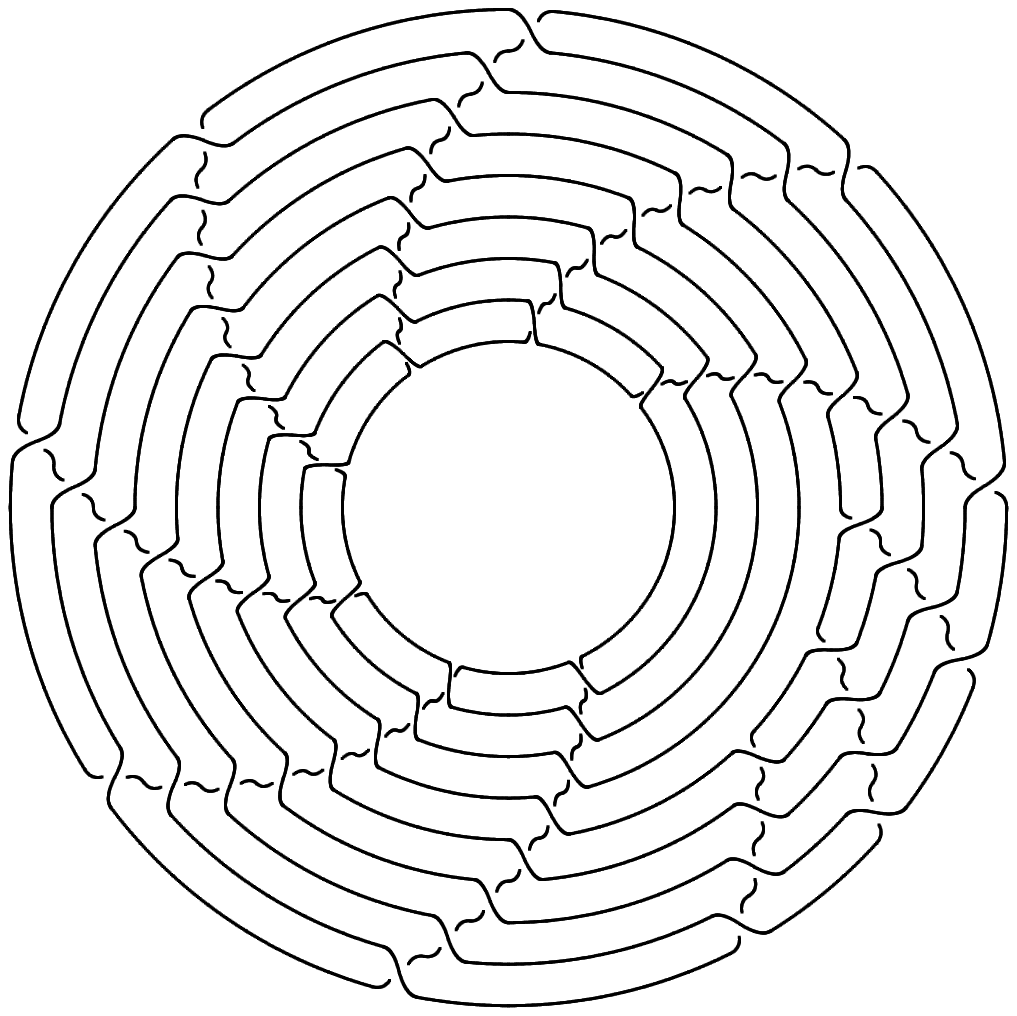}}
\qquad\quad
    \subfigure[{\tiny $T((9,7),(5,3),(3,4))$}]{\includegraphics[scale=0.37]{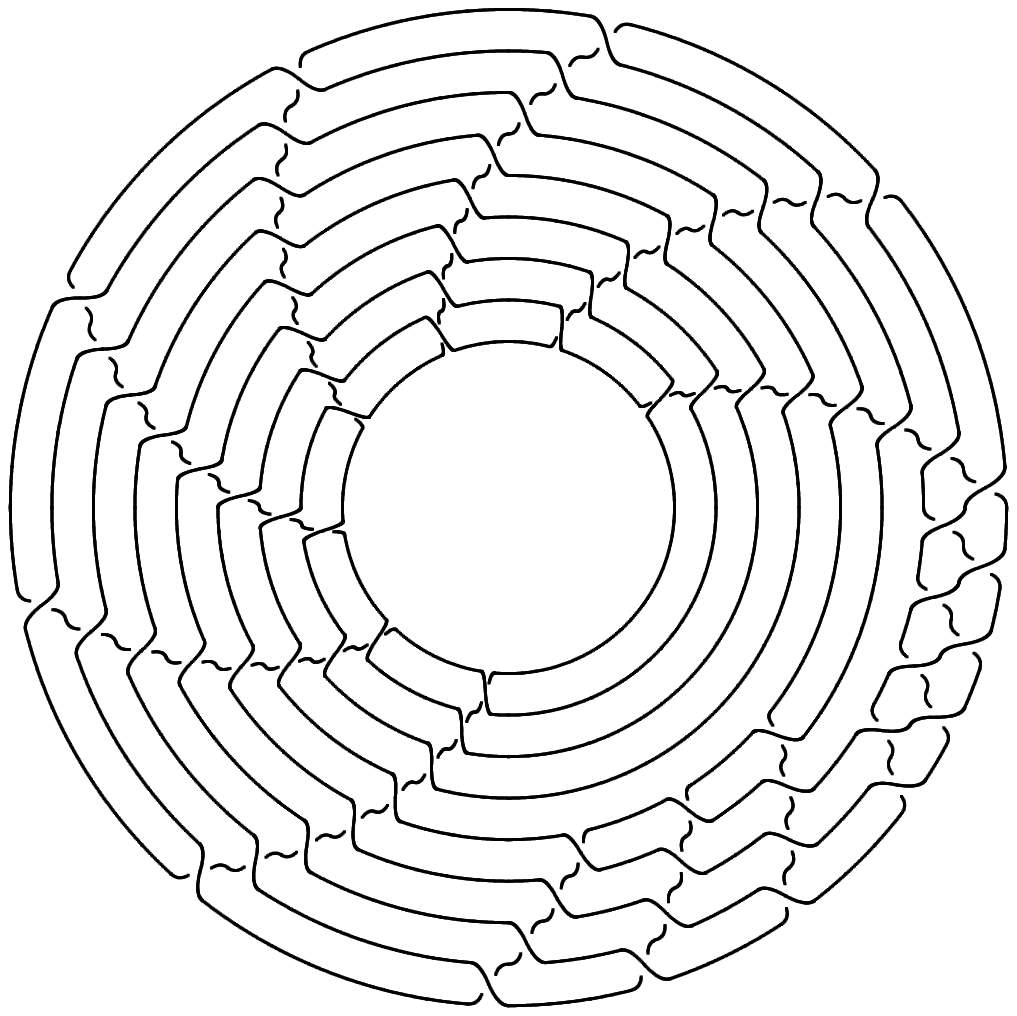}}
  \end{center}
  \caption{A torus knot, a twisted torus knot, and a $T$--knot using standard roots.}
  \label{fig:T-knots}
\end{figure}

The above definition of $T$--links is generalized from the definition
in \cite{birman-kofman}, but our $s_i$'s may be negative and our
$r_i$'s form a decreasing sequence.  By [4], it follows that positive
$T$--links using only the roots $\bar\delta_{r_i}$ coincide with
Lorenz links.  The $T$--link point of view opens the door to
understanding the geometry of Lorenz link complements.  But our
results below do not require positivity.  We provide volume bounds for
$T$--links, so in particular, for Lorenz links.


\subsection{Volume bounds}

Let $\vol(K)$ denote the volume of the link complement $S^3\setminus
K$.  If $K$ is not hyperbolic, $\vol(K)$ is the sum of volumes of the
hyperbolic pieces of $S^3\setminus K$.  Let $v_3 \approx 1.0149$
denote the volume of the regular hyperbolic ideal tetrahedron.  In
Section \ref{sec:upper-bound}, we prove the following result.

\begin{theorem}\label{thm:s-bound}
 Let $T(p,q,r,s)$ be a twisted torus link with positive root $\beta$. Then
\begin{align*}
\vol(T(p,q,r,s))\ & < \ 10v_3 &  {\it if\ } r=2, \\
\vol(T(p,q,r,s))\ & < \ v_3(2r+10) &  {\it if\ } s\ {\rm mod}\ r=0, \\
\vol(T(p,q,r,s))\ & < \ v_3(r^2 + r + 10) &  {\it if\ } \beta = \delta_r {\it \ or\ } \bar\delta_r, \\
\vol(T(p,q,r,s))\ & < \ v_3( r^2 + 4r + 4) & {\it otherwise.}
\end{align*}
\end{theorem}

\begin{theorem}
  \label{thm:sharp2}
Choose any sequence $(p_N, q_N) \to (\infty,\infty)$, such that $\gcd(p_N, q_N)=1$. Then the twisted torus knots
$T(p_N,q_N,2,2N)$ have volume  approaching $10v_3$ as $N \to \infty$.
\end{theorem}

\noindent
The noteworthy feature of Theorem \ref{thm:s-bound} is that the upper bound only depends on the parameter $r$. The independence of $p$ and $q$ was a surprise to the authors.  
One consequence of this independence is that there is no direct relationship between the braid index and volume of a link.

\begin{corollary}
\label{cor:braid-vs-volume}
  The twisted torus knots $T(p,q,2,s),\ p,q>2,\ s>0$, have arbitrarily
  large braid index and volume bounded by $10v_3$.
\end{corollary}

\begin{proof}
  For $q, s>0$ and $p>r$, the minimal braid index of $T(p,q,r,s)$ with
  root $\delta_r$ or $\bar\delta_r$ is exactly $\min(p,q)$ if $r\leq q$, and
  $\min(s+q,r)$ if $r \geq q$ (Corollary 8 of \cite{birman-kofman}).
  By Theorem \ref{thm:s-bound}, any $T(p,q,2,s)$ has volume bounded
  above by $10v_3$, but for all $p,q>2$, its minimal braid index is
  $\min(p,q)$.
\end{proof}

The reverse of Corollary \ref{cor:braid-vs-volume} is also true: for
example, closed $3$--braids have unbounded hyperbolic volume. See  \cite[Theorem 5.5]{fkp:farey}.

When the twisted torus knots are Lorenz, we can use Lorenz duality
\cite[Corollary 4]{birman-kofman} to obtain another volume bound in terms of $q$.

\begin{corollary}\label{cor:duality}
If we use the roots $\delta_r$ or $\bar\delta_r$, and let $(q\cdot s)>0$ and $p > r$, 
\begin{align*}
\vol(T(p,q,r,s))\ & < \ 10v_3 &  {\it if\ } q =2, \\
\vol(T(p,q,r,s))\ & < \ v_3(2|q|+10) &  {\it if\ } p\ {\rm 
mod}\ q=r, \\
\vol(T(p,q,r,s))\ & <  \ v_3 ( q^2+|q|+10 ) & 
{\it otherwise.} 
\end{align*}
\end{corollary}

\begin{proof}
The braids $(\delta_p^q\, \delta_r^s), (\bar\delta_p^q\,
\bar\delta_r^s), (\delta_r^s \,\delta_p^q)$ and $(\bar\delta_r^s
\,\bar\delta_p^q)$ have isotopic closures.  When $(q \cdot s) >0$, all
the twisting is in the same direction.  Thus, under the hypotheses of
the corollary, the twisted torus links are Lorenz or mirror images of
Lorenz links.  So they satisfy the following duality coming from the
symmetry of the Lorenz template \cite{birman-kofman}:
$$ T(p,q,r,s) = T(q+s, r, q, p-r) \quad {\rm if}\ q,s >0. $$
In particular, $q$ and $r$ can be exchanged.
\end{proof}

Our results extend to volume bounds for $T$--links.
\begin{theorem}\label{thm:T-bound}
 If $L$ is the  $T$--link  
$T((p,q),(r_1, s_1,\beta_1), \dots, (r_k,s_k,\beta_k))$,  
\begin{align*}
\vol(L)\ & < \ v_3 \left(r_1^2+9r_1-8 \right) & \quad {\it if\ all}\ s_i\ {\rm mod}\ r_i=0, \\
\vol(L)\ & < \ v_3 \left( \tfrac{1}{3} r_1^3 + \tfrac{5}{2} r_1^2 + 5  r_1 -5 \right) & \quad {\it otherwise.}
\end{align*}
\end{theorem}

Again, the notable feature of Theorem \ref{thm:T-bound} is that even
though it takes many parameters to specify a $T$--link, a single
coordinate suffices to bound the volume from above.  

If $p>r_1$, the braids $(\bar\delta_{p}^{q}\, \bar\delta_{r_1}^{s_1} \cdots
\bar\delta_{r_k}^{s_k})$ and $(\delta_{r_k}^{s_k} \cdots
\delta_{r_1}^{s_1}\, \delta_{p}^{q})$ have isotopic closures.  So
if all $s_i>0$ and all roots are $\bar\delta_{r_i}$, these T-links are
Lorenz by \cite{birman-kofman}, and Lorenz duality implies a result
analogous to Corollary \ref{cor:duality}, with $r_1$ replaced by
$(q+s_1+s_2+\ldots +s_{k-1})$.

If $r_1\leq d=\gcd(p,q)$ then $T((p,q),(r_1, s_1), \dots, (r_k,s_k))$
is a satellite link with companion $T(p/d, q/d)$.  In the JSJ
decomposition of this link complement, only the solid torus minus
$T((r_1, s_1), \dots, (r_k,s_k))$ can have non-zero volume, which is
bounded by the function from Theorem \ref{thm:T-bound} with $r_1$
replaced by $r_2$.  Similarly, if $r\le \gcd(p,q)$ in Theorem
\ref{thm:s-bound} the volume is zero. So we assume from now on that
$r_1>\gcd(p,q)$, resp.\ $r>\gcd(p,q)$.

In Section \ref{sec:lower-bound}, we prove the following theorem,
which shows that these volume bounds are non-trivial.

\begin{theorem}
\label{thm:lower-bound}
For any number $V$, there exists a hyperbolic twisted torus knot whose complement
has volume at least $V$.
\end{theorem}

In the proof of Theorem \ref{thm:lower-bound}, we construct links by
twisting along annuli.  This theorem can be compared with related
work of Baker \cite{baker}.  He showed that the class of Berge knots,
which contains some twisted torus knots, contains knots which have
arbitrarily large volume.  However, Baker's examples are not
necessarily twisted torus knots.  Nor are the examples produced to
prove Theorem \ref{thm:lower-bound} necessarily Berge knots.

\subsection{Acknowledgments}  We thank the organizers of the 2007
conference at LSU, {\em A second time around the Volume Conjecture},
where this work was started.  We thank the anonymous referee for the
insightful idea that dramatically improved the tetrahedron count of Lemma  \ref{lemma:tetrbound2}.
We thank Joan Birman for helpful comments.  We are grateful for NSF
support, and support for the first and third author by PSC-CUNY.

\section{Positive roots of the full twist in $B_n$}\label{sec:roots}

Twisting is a natural geometric operation on links because any full
twist comes from $\pm 1$ Dehn surgery on an unknot in $S^3$.  However,
to define a twisted torus link without full twists, as in Definition
\ref{def:ttk}, we must first choose a particular root of the full
twist.  The link type generally changes with different choices of
roots.  Braids provide a natural notation for this choice, and in this
section we give an elementary proof of the classification of braid
isotopy classes of roots in Theorem \ref{thm:roots}.  This result also
follows from Corollary 12 of \cite{Birman-Gebhardt-Meneses}, which
uses the Garside structure of the braid group.


Recall that a positive root $\beta$ is a positive $n$--braid whose
$n$--th power is $\Delta^2$ in $B_n$.  Since $\Delta^2$ has length
$n(n-1)$ in the braid generators, a positive root must have length
$n-1$.  Moreover, since all $n-1$ generators must be involved, $\beta$
must have the form
$$\beta=\sigma_{i_1}\dots\sigma_{i_{n-1}}$$ with $i_1,\dots,i_{n-1}$ a
permutation of $1,\dots,n-1$.  This list of $(n-1)!$ braid words
includes all positive roots, but with multiplicity because many of
these are isotopic braids.

\begin{theorem}\label{thm:roots}
There are $2^{n-2}$ braid isotopy classes of positive $n$-th roots,
all of which are conjugate in $B_n$ to the standard root
$\delta_n=\sigma_1\cdots \sigma_{n-1}$.
\end{theorem}

\begin{table} 
\caption{All eight 5-th roots, repeated three times in each diagram.
  The corresponding braidword and subset $\{j_1,\dots,j_r\}$ are given
  below each one.}
\label{roots_table}
\begin{center}
\begin{tabular}{cccccccc}
\includegraphics[height=0.8in]{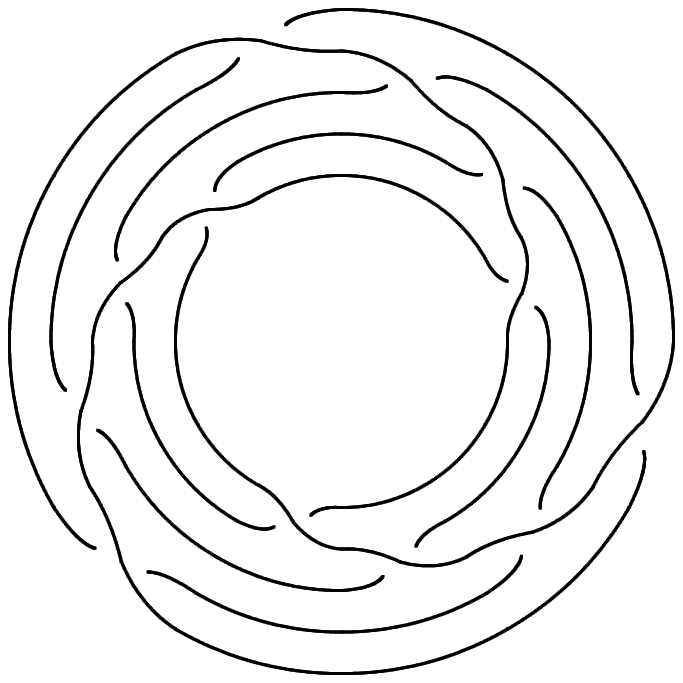} &
\includegraphics[height=0.8in]{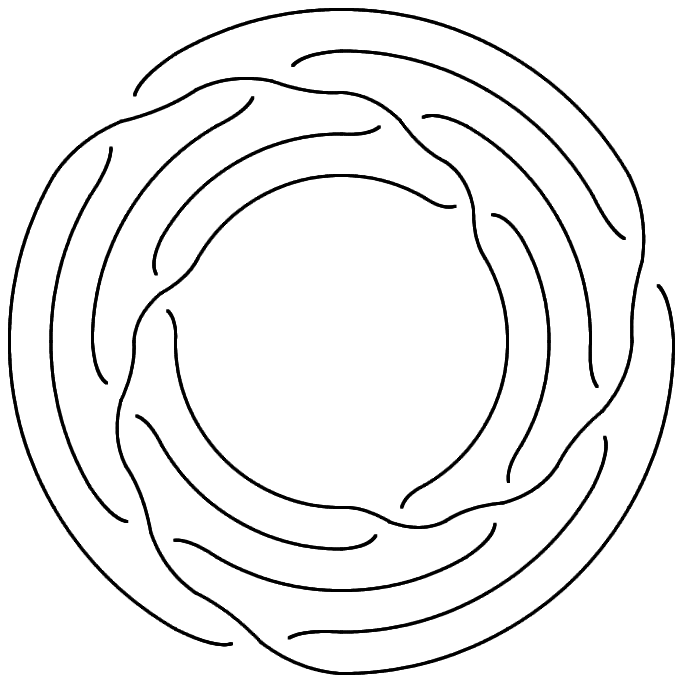} &
\includegraphics[height=0.8in]{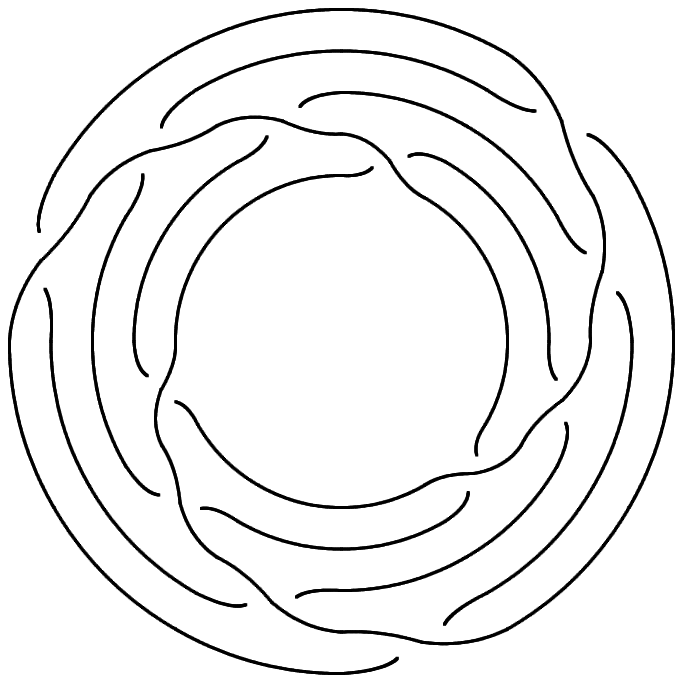} &
\includegraphics[height=0.8in]{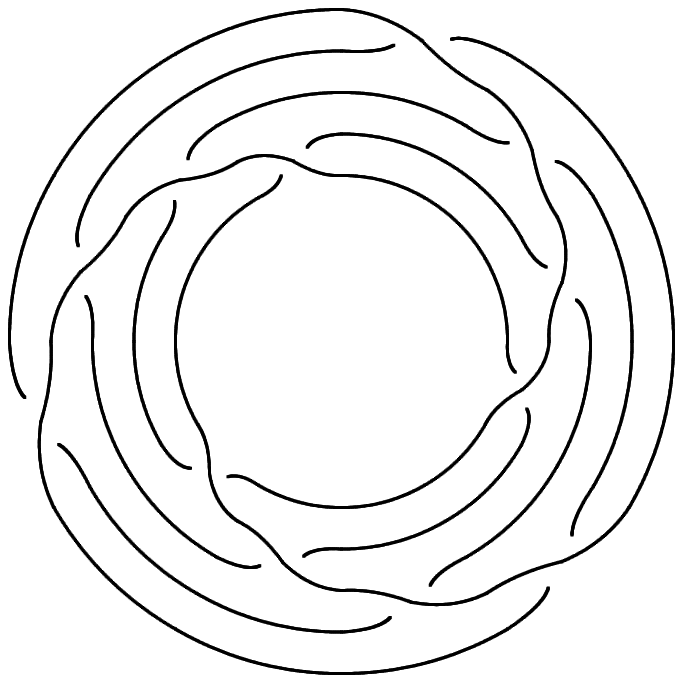} \\
4321 & 1432 & 2143  &3214 \\
$\{ \emptyset \}$ & $\{ 1 \}$ & $\{ 2 \}$ & $\{ 3 \}$ \\
\\ 
\hline \\ 
\end{tabular}

\begin{tabular}{cccccccc}
\includegraphics[height=0.8in]{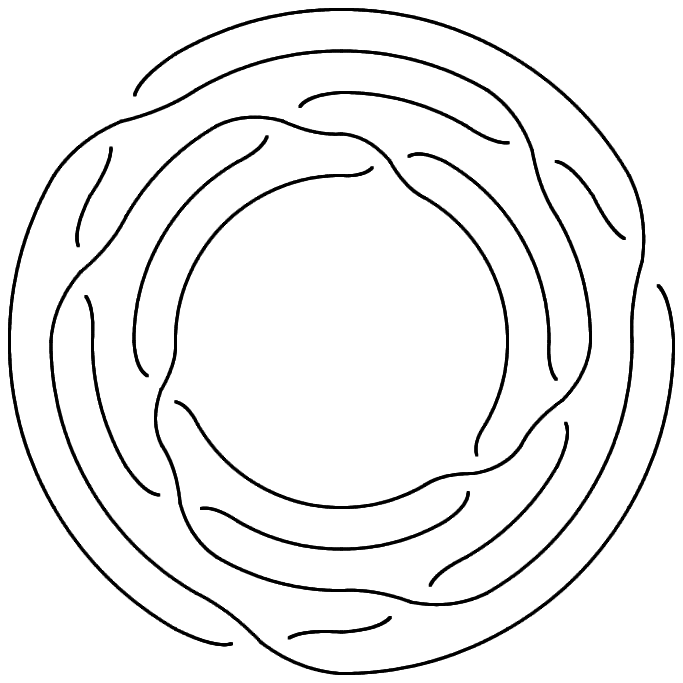} &
\includegraphics[height=0.8in]{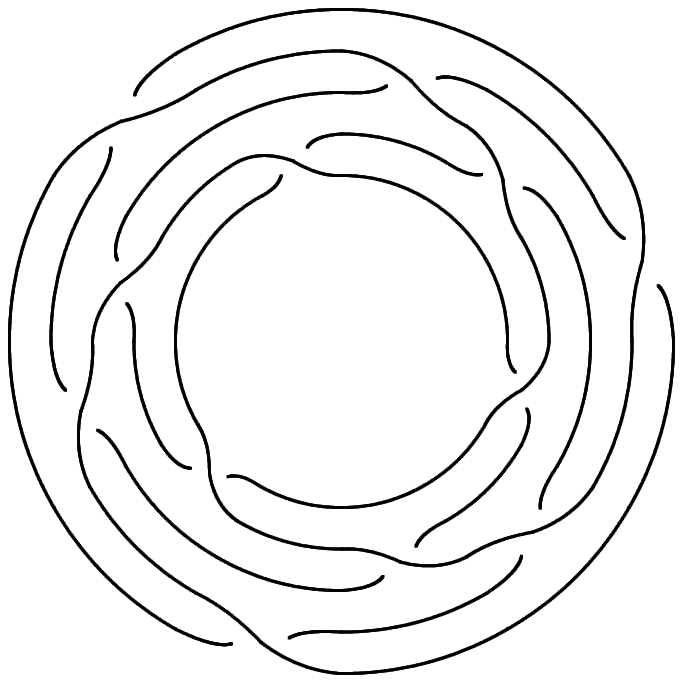} &
\includegraphics[height=0.8in]{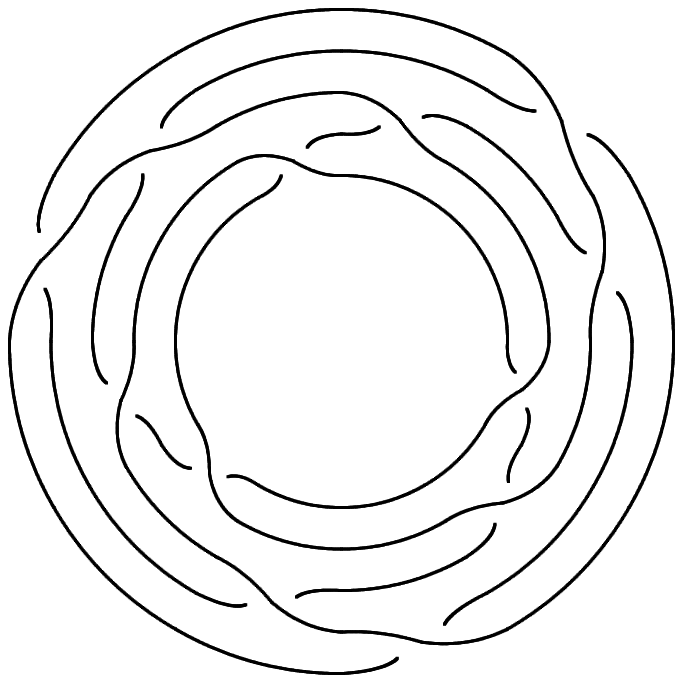} &
\includegraphics[height=0.8in]{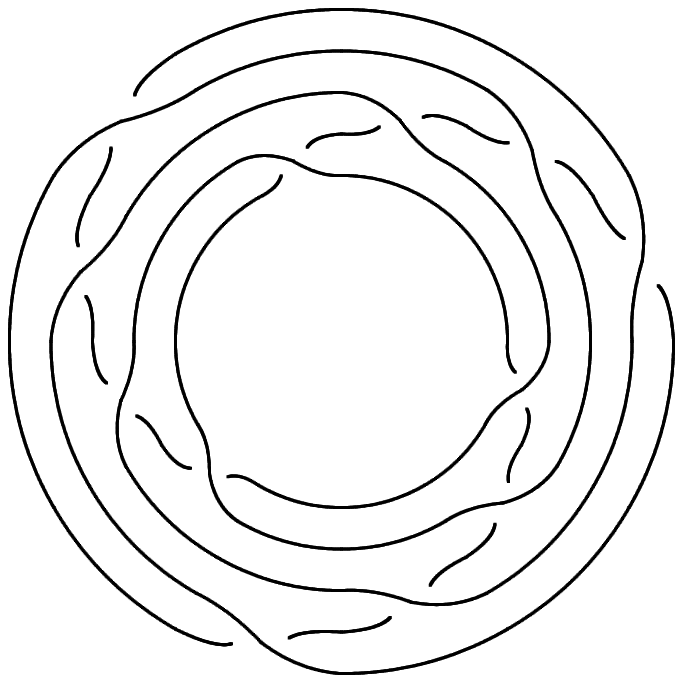} \\
1243&  1324&  2134 & 1234 \\

$\{ 1,2 \}$ & $\{ 1,3 \}$ & $\{ 2,3 \}$ & $\{ 1,2,3 \}$ \\
\end{tabular}
\end{center}
\end{table}

\begin{proof}
Using the fact that $\sigma_i\sigma_j=\sigma_j\sigma_i$ whenever
$|i-j|\geq 2$, we can move any $\sigma_{i_k}$ with $i_k<i_{k-1}$ to
the left in the above expression unless $i_k=i_{k-1}-1$. After
repeating this until no further such moves are possible, the indices
$i_1, i_2,\dots i_{n-1}$ will form the concatenation of some number
$(r+1)$ of monotone decreasing chains,
\begin{equation}\label{eq:normal-braid}
  j_1,j_1-1,\dots,1; j_2,j_2-1,\dots,j_1+1; j_3,j_3-1, \dots, j_2+1;
  \dots; n-1, n-2,\dots, j_r+1.
\end{equation}
This gives a normal form for $\beta$ (up to braid isotopy) that is
determined by $\{j_1,j_2,\dots,j_r\}$, which is a subset of
$\{1,\dots,n-2\}$.  Every subset corresponds to a normal form
expression, so there are $2^{n-2}$ possible normal form expressions.
(The empty set corresponds to the chain $n-1, n-2,\dots, 1$.)

We claim that these normal forms give $2^{n-2}$ different
roots. Indeed, a simple calculation shows that their images in the
permutation group $S_n$ are distinct, so they are distinct.

Now, identify braids in this list if they are equivalent under cyclic
permutation plus the braid relation $\sigma_i\sigma_j =
\sigma_j\sigma_i$ for $|i-j|\geq 2$. It is easy to see that every
braid given in the normal form is cyclically equivalent to either
$\delta_n$ or $\bar\delta_n = \sigma_{n-1}\cdots \sigma_1$.  Moreover, $\bar\delta_n$ is
conjugate to $\delta_n$ by $\Delta$.  Therefore, every braid given by
this normal form is conjugate to $\delta_n$, and hence is an $n$--th
root.
\end{proof}

From a picture, it becomes clear that these braids are roots.  We
illustrate this in Table \ref{roots_table}.  In this table, braids
should be read counterclockwise, starting at $120^\circ$ on the circle.

The following lemma is immediate from the normal form for $\beta$,
particularly the form of the indices $i_1, i_2, \dots, i_{n-1}$ in
equation \eqref{eq:normal-braid}.  We record it here, since we will
use it in the next section.

\begin{lemma}\label{lemma:normal-form}
In the normal form for $\beta=
\sigma_{i_1}\sigma_{i_2}\dots\sigma_{i_{n-1}}$:
\begin{enumerate}
\item[(a)] The generator $\sigma_j$ appears before $\sigma_{j-1}$ and
  $\sigma_{j+1}$, $1< j< n-1$, if and only if the index $j$ is the
  first entry of a decreasing chain of length at least two in equation
  \eqref{eq:normal-braid}.
\item[(b)] Similarly, $\sigma_j$ appears after both $\sigma_{j-1}$ and
  $\sigma_{j+1}$ if and only if the index $j$ is the last entry of a
  decreasing chain of length at least two in equation
  \eqref{eq:normal-braid}.
\item[(c)] The generator $\sigma_1$ appears before $\sigma_2$ if and
  only if $\sigma_1$ is the first generator of the braid word, with
  the index $1$ forming a chain of length one.  The generator
  $\sigma_1$ appears after $\sigma_2$ if and only if the index $1$ is
  the last entry of a decreasing chain of length at least two in
  equation \eqref{eq:normal-braid}.
\item[(d)] The generator $\sigma_{n-1}$ appears before $\sigma_{n-2}$ if
  and only if the index $(n-1)$ is first in a chain of length at least
  two; and $\sigma_{n-1}$ appears after $\sigma_{n-2}$ if and only if
  $\sigma_{n-1}$ is the last generator in the braid word, with
  the index $(n-1)$ forming a chain of length one.\qed
\end{enumerate}
\end{lemma}

\section{Upper volume bounds}\label{sec:upper-bound}

In this section, we prove Theorems \ref{thm:s-bound}, \ref{thm:sharp2} and
\ref{thm:T-bound}, establishing upper bounds on the volumes of twisted torus links and $T$--links.

\subsection{Twisted torus links}
First, we define $M(p,q,r,s)$, which is a surgery parent manifold to $T(p,q,r,s)$.

\begin{define}
\label{def:mpqr}
Start with integers $p,q>0$ and $r$ so that $0<r\leq p+q$.  Let
$C_1$ be an unknot in $S^3$, so $S^3 \setminus C_1$ is the solid torus
$S^1 \times D^2$, where we view $D^2$ as the unit disk.  Let $C_2$ be
the core of this solid torus, $C_2=S^1\times \{0\}$.  Let $T(p,q)$ be
the $(p,q)$--torus link sitting on $S^1 \times S^1_{1/2}$, where
$S^1_{1/2}$ is the circle at radius $1/2$ in the disk $D^2$.  Augment
$T(p,q)$ with an unknotted circle $L$ that encircles the first $r$
strands.  For any given positive root $\beta \in B_r$, and integer
$s$, replace the $r$ strands of $T(p,q)$ encircled by $L$ with the
braid $\beta^{s}$, and call the result $K(p,q,r,s, \beta)$.  Let
$M(p,q,r,s)$ be the link complement $S^3 \setminus (C_1 \cup C_2 \cup
L \cup K(p,q,r,s,\beta))$.  (As usual, we suppress $\beta$.)
\end{define}

Observe the following facts about $M(p,q,r,s)$.  Let $I=(-1,1)$.
First, since $C_1 \cup C_2$ is the Hopf link, $S^3 \setminus (C_1 \cup
C_2) \cong T^2 \times I$.  Hence, $M(p,q,r,s)$ is homeomorphic to the
complement of a link of at least two components in $T^2 \times I$, one
component corresponding to $L$, and the others to $K(p,q,r,s,\beta)$
of Definition \ref{def:mpqr}.  We will illustrate examples of
$M(p,q,r,s)$ by drawing links in $T^2 \times I$. We will also use the
framings induced from $T^2 \times \{0\}$ on $C_1$ and $C_2$.

Second, the twisted torus link $T(p,q,r,s)$ is obtained from
$M(p,q,r,s)$ by Dehn filling along slopes $(0,1)$ and $(1,0)$ on $C_1$
and $C_2$, respectively, and along the slope $(1,0)$ on $L$.  More generally, we
obtain a twisted torus link of the form $T(p',q',r, Nr+s)$ by Dehn
filling $L$ along the slope $(1,N)$ for any integer $N$, and Dehn
filling $C_1$ and $C_2$ along slopes with geometric intersection number $1$.  
For example, see
\cite{Rolfsen}.  Since volume only decreases under Dehn filling
\cite{Thurston}, we wish to bound the volume of $M(p,q,r,s)$.

Finally, the homeomorphism type of $M(p,q,r,s)$ is summarized by the
following lemma.

\begin{lemma}
\label{lemma:Mpqr-homeo}
  For $r>\gcd(p,q)$, $M(p,q,r,s)$ is homeomorphic to $M(n,m,r,s')$ or
  $M(m,n,r,s')$, where $s' = s \mod r$, $0\leq s' < r$, and $n$ and
  $m$ come from a truncated continued fraction expansion of $p/q$.
  Precisely, $n$ and $m$ are integers satisfying $0<n<r$,
  $0<m<r$, $n+m\geq r$, and
$$ \frac{p}{q} = a_0 + \cfrac{1}{a_1 + \cfrac{1}{\ddots \cfrac{1}{a_k + m/n}}}  \: .$$
where $a_i$ are positive integers for $1 \leq i \leq k$. 
\end{lemma}

\begin{proof}
First, since $L$ is an unlink encircling the $r$ strands of the braid
$\beta^s$, $M(p,q,r,s)$ is homeomorphic to $M(p,q,r,s+jr)$ for any
integer $j$.  Thus in particular it is homeomorphic to $M(p,q,r,s')$
where $s'= s \mod r$ and $0\leq s' <r$.

Next, since $M(p,q,r,s)$ is a link complement in $T^2\times I$, it
will be homeomorphic to the link complement obtained by Dehn twisting
an integer number of times about the meridian or longitude of $T^2$.
Thus applying the (truncated) Euclidean algorithm to the slope $p/q$
on $T^2$, we may reduce the slope to some $m/n$ with $m<r$, $n<r$,
$m+n\geq r$.  The process gives the truncated continued fraction
expansion of $p/q$.
\end{proof}

Figure \ref{fig:Mnmr} illustrates the proof of Lemma
\ref{lemma:Mpqr-homeo} when $p=3$, $q=7$, $r=5$, and $s=0$.  In that
figure, Dehn twist about the horizontal curve $(0,1)$, and isotope to
obtain $M(3,4,5,0)$.

\begin{figure}
  \begin{center}
    \includegraphics{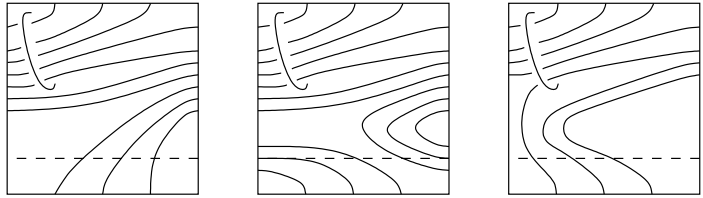}
  \end{center}
  \caption{$M(3,7,5,0)$ is homeomorphic to $M(3,4,5,0)$.}
  \label{fig:Mnmr}
\end{figure}

Our goal is to bound the simplicial volume of $M(p,q,r,s)$. Recall that for
any compact $3$--manifold $M$, whether closed or with boundary, the
\emph{simplicial volume} is defined to be ${\rm Vol}(M)= v_3\, || M ||$, where
$|| M ||$ is the Gromov norm of $M$. See \cite[Chapter 6]{Thurston} 
for background on the Gromov norm. For our purposes, we will need three properties:
\begin{itemize}
\item $|| M ||$ is bounded above by the number of (compact or ideal) tetrahedra needed to triangulate $M$.
\item $|| M ||$ is non-increasing under Dehn filling.
\item If $M$ is hyperbolic, then $v_3\, || M ||$ is the hyperbolic volume of $M$. Thus there is no ambiguity in the notation $\vol(M)$.
\end{itemize}

Combining these properties, we conclude that $v_3$ times the number of ideal tetrahedra in a triangulation of $M(p,q,r,s)$ provides an upper bound on the volume of any of its hyperbolic Dehn fillings. This upper bound applies regardless of whether $M(p,q,r,s)$ is hyperbolic.

The following
straightforward lemma will assist us in counting the number of
 tetrahedra in an ideal triangulation of $M(p,q,r,s)$.

\begin{lemma}
\label{lemma:subdivide}
Two pyramids glued along a common base, a polygon with $t$ sides, may
be subdivided into $t$ tetrahedra.  In the case $t=3$, the pyramids
may be subdivided into $2$ tetrahedra.
\end{lemma}

\begin{proof}
  If $t=3$, each pyramid is a tetrahedron, and there is nothing to
  prove.  If $t>3$, then remove the base polygon, and add an edge
  running between the two tips of the pyramids.  Perform stellar
  subdivision, obtaining one tetrahedron for each of the $t$ edges of
  the base polygon.
\end{proof}

\begin{lemma}
\label{lemma:tetrbound1}
Let $n$, $m$, and $r$ be integers such that $n<r$, $m<r$, and $n+m\geq
r$.  The manifold $M(n,m,r,0)$ can be decomposed into $t$ ideal
tetrahedra, where 
\begin{equation*}
t=
\begin{cases}
  10 & \text{if $r=2$,} \\
  2r+8 & \text{if $n+m=r$ and $r>2$,} \\
  2r+10 & \text{otherwise.}
\end{cases}
\end{equation*}
\end{lemma}

\begin{proof}
Recall $M(n,m,r,0)$ is the complement of a link in $T^2\times (-1,1)$,
with the torus link $T(n,m)$ on $T^2 \times \{0\}$.  Put $L$
perpendicular to $T^2\times \{0\}$, meeting it transversely in two
points.

First, cut along the torus $T^2 \times \{0\}$, as shown in Figure
\ref{fig:twtorus-decomp}\,(left).  This divides the manifold into two
pieces.  By cutting along the torus, the disk bounded by the component
$L$ has been cut into two.  Slice up the middle of each half--disk and
open it out, flattening it onto $T^2 \times \{0\}$, as in Figure
\ref{fig:twtorus-decomp}\,(middle). Each half--disk has been sliced
open into two parts, each of which is an $(r+1)$--gon.  In Figure
\ref{fig:twtorus-decomp}\,(right), an example is shown for $r=5$.

\begin{figure}
	\includegraphics{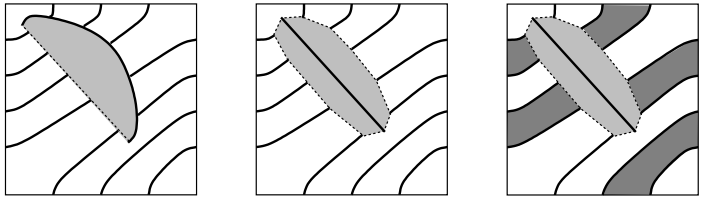}
	\caption{To decompose $M(p,q,r,0)$ into tetrahedra, a
          half-disk is sliced and flattened onto $T^2 \times \{0\}$.
          Thick lines are components of the link.  Dotted lines are
          edges of a polyhedral decomposition.  One or two polygons remain
          after collapsing the shaded bigons.}
\label{fig:twtorus-decomp}
\end{figure}

Outside the flattened half--disk, regions are bigons, which we collapse to
ideal edges as in Figure 4 (right), and either two quadrilaterals or a
single hexagon adjacent to the ends of the half-disks.
This can be seen as follows.  View $T^2\times \{0\}$ as a rectangle
with the usual side identifications to form the torus.  Any region $U$
outside the half--disk (as in Figure
\ref{fig:twtorus-decomp}\,(middle)) will either meet a single edge of
the half--disk, two edges if $U$ is adjacent to an end of the
half--disk, or zero edges if $U$ is not adjacent to the half--disk.
Such regions are glued according to the identifications for a torus.
Consider the regions meeting ends of the half--disk.  Each of these
two regions is glued to exactly two other regions.  Either each of the
two regions glue up to regions meeting only a single edge of the
half--disk, in which case both regions are quadrilaterals, or the two
regions may glue to meet each other.  In this second case, the regions
must additionally glue to regions meeting exactly one edge, so we have
a hexagon.

The case of two quadrilaterals is the case of Figure
\ref{fig:twtorus-decomp}.  Since we obtain a hexagon only when both
ends of the flattened half-disk belong to the same complementary
region, this occurs if and only if $n+m=r$.  Finally, any remaining
regions meet one or zero edges, and must glue to form bigons; two
bigons are shaded dark in Figure \ref{fig:twtorus-decomp}\, (right).

Now cone each sliced half--disk, quadrilateral, and hexagon to the
boundary component $T^2\times \{1\}$, as well as to the boundary
component $T^2\times \{-1\}$.  This gives a decomposition of the
manifold $M(n,m,r,0)$ into pyramids.  Since the half disks are
identified to each other and the outside regions are identified along the
two pieces, the pyramids are glued in pairs along the regions on $T^2\times
\{0\}$.  By Lemma \ref{lemma:subdivide}, we may subdivide into:
\begin{itemize}
  \item $2(r+1)$ tetrahedra for two pairs of $(r+1)$--gons if $r>2$.
    If $r=2$, improve this to $4$ tetrahedra.
  \item $6$ tetrahedra for the single hexagon, if $n+m=r$.  Otherwise,
    $8$ tetrahedra for two quadrilaterals.
\end{itemize}
Observe that if $r=2$, then $n=m=1$ so we will have a hexagon in this
case.  Adding together these counts gives the result.
\end{proof}

\begin{lemma}
\label{lemma:tetrbound2}
  Let $r>2$.  Let $n$, $m$, $r$ and $s$ be integers such that $n<r$,
  $m<r$, $n+m\geq r$, and $0<s<r$.  For any positive root $\beta\in
  B_r$, the manifold $M(n,m,r,s)$ can be decomposed into at most $t$ ideal
  tetrahedra, where
\begin{equation*}
t=
\begin{cases}
  rs+3r-s+9 & \text{if $\beta=\delta_r$ or $\bar\delta_r$,} \\
  rs+6r-s+3& \text{otherwise.}
\end{cases}
\end{equation*}
\end{lemma}

\begin{proof}
This proof is similar in spirit to that of Lemma \ref{lemma:tetrbound1}:  we subdivide $M(n,m,r,s)$ into two ``polyhedral'' pieces along a $2$--complex roughly corresponding to the projection torus $T^2 \times \{0\}$. Each of these pieces can be subdivided into tetrahedra by coning to $T^2 \times \{ \pm 1 \}$. 

As above, the disk bounded by $L$ gets cut into two,
each half--disk gets sliced up the middle and flattened out into two
$(r+1)$--gons on $T^2 \times \{0\}$. However, when $s \neq 0$, we now have a braid $\beta^s$, which can be
seen as crossings on the torus $T^2 \times \{ 0 \}$. Encircle the braid $\beta^s$ in
$T^2\times \{0\}$ with a simple closed curve $\gamma$, separating this braid
from the rest of the diagram. (See Figure \ref{fig:s-decomp}.) 
Note that the diagram outside
this curve $\gamma$ agrees with the diagram of $M(n,m,r,0)$. From here, the argument will proceed in four steps:

\begin{enumerate}
\item[Step 1.] Count and characterize the polygons inside $\gamma$. This is done in Lemma \ref{lemma:braid-quads}.

\item[Step 2.] At each quadrilateral bounded by the projection of $\beta^s$, insert a \emph{medial tetrahedron}, as in Figure \ref{fig:inner_tet}. 

\item[Step 3.] Glue the faces of medial tetrahedra to certain adjacent triangles. This ``collapsing'' process, carried out in Lemma \ref{lemma:collapse} (see Figures \ref{fig:sailboat} and \ref{fig:inner_tet}), reduces the number of faces visible from $T^2 \times \{ \pm 1 \}$.

\item[Step 4.] Cone all remaining faces to $T^2 \times \{ \pm 1 \}$, and count the resulting tetrahedra. This will complete the proof.

\end{enumerate}

\begin{figure}
  \begin{tabular}{ccc}
  \begin{overpic}[height=1.5in]{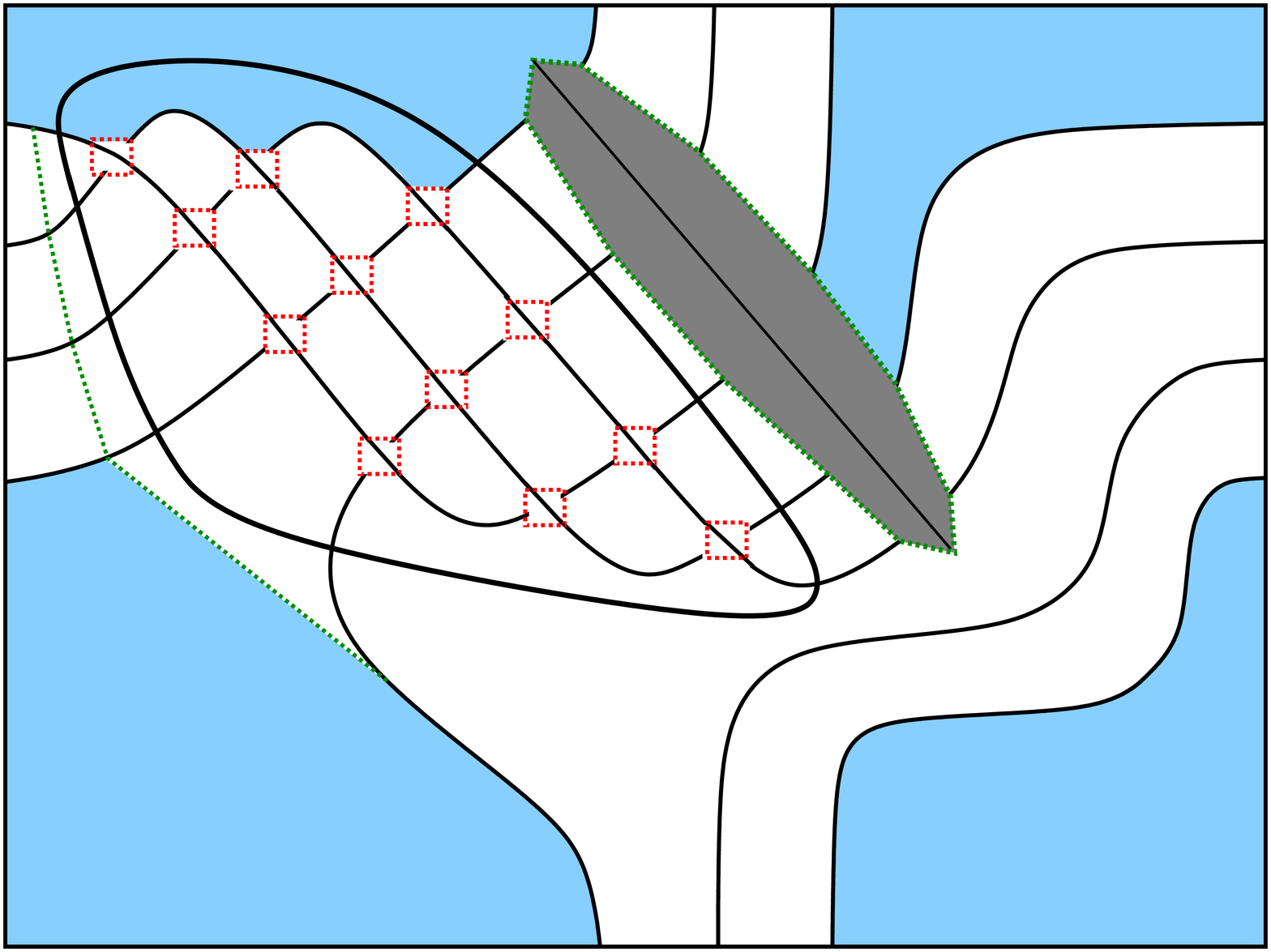}
  \put(40,22){$\gamma$}
  \end{overpic}
  \ \ &
\ \ 	\begin{overpic}[height=1.5in]{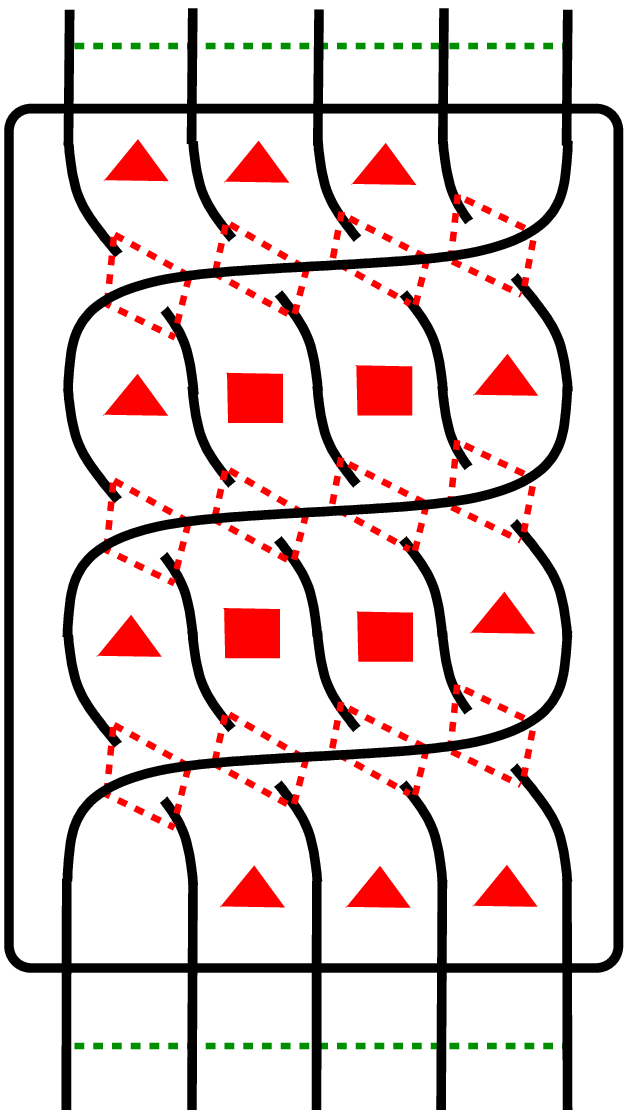} 
  \put(60,80){$\gamma$}
\end{overpic}
\  \ \  &  \  \includegraphics[height=1.5in]{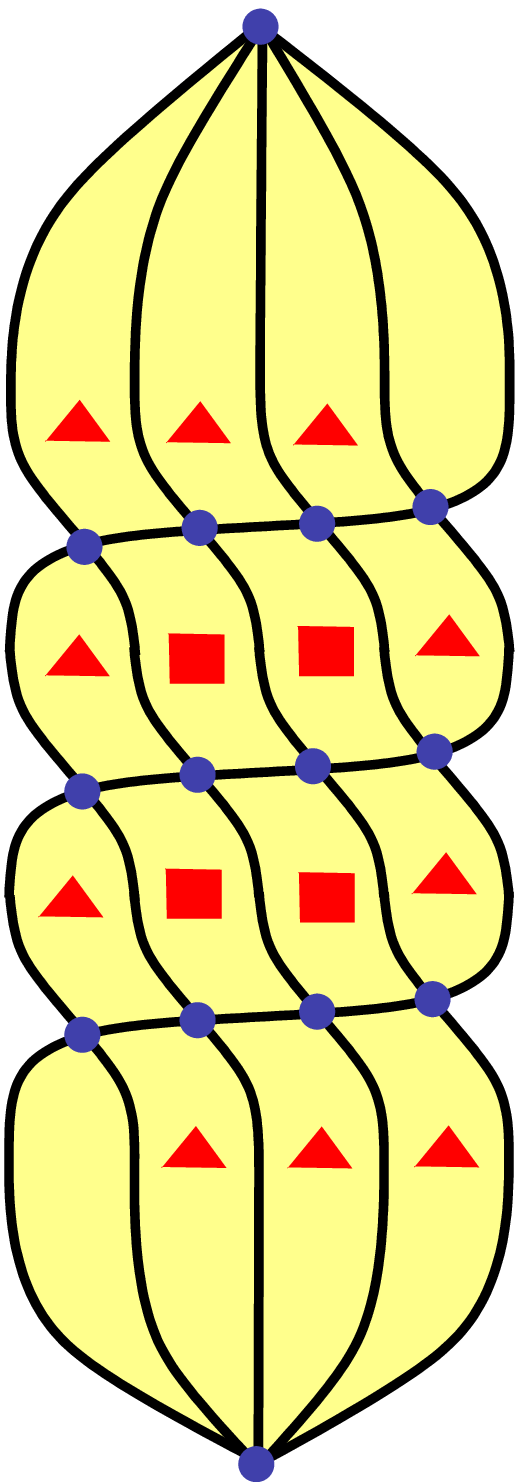} \\
    (a) & (b) & (c)
  \end{tabular}
  \caption{(a) Decompose $M(n,m,r,s)$ into tetrahedra, here $r=5$,
    $s=3$ and $\beta=\bar\delta_5$.  Thick lines are link components, dotted
    lines are polygon edges.  This figure differs from Figure
    \ref{fig:twtorus-decomp}\,(middle) only in the encircled
    region. (b) For $\beta=\bar\delta_5$, schematic figure for
    $\beta^3$ with triangles and quadrilaterals indicated. (c) The
    disk triangulated with the projection of $\beta^3$. }
\label{fig:s-decomp}
\end{figure}

To begin the argument, we add an edge for every crossing of $\beta^s$. 
We also add edges, called {\em peripheral edges}, at the start and end of the
braid, as in Figure \ref{fig:s-decomp}\,(b). 
Every face bounded by these edges must be a quadrilateral, a triangle, or a bigon.
No other polygons can occur because
$\beta=\sigma_{i_1}\sigma_{i_2}\dots\sigma_{i_{n-1}}$ is a word formed
by a permutation of the $n-1$ generators $\sigma_1, \dots,
\sigma_{n-1}$, hence in the word $\beta^s$, no generator $\sigma_j$
appears twice before a single appearance of $\sigma_{j+1}$ or
$\sigma_{j-1}$.
Moreover, bigons can only occur adjacent to the start or end of the braid.

Quadrilaterals and triangles that contain peripheral edges (i.e., that
are adjacent to the start or end of the braid) will be called {\em
  peripheral} faces, and otherwise they will be called {\em inner}
faces.  Inner triangles can only occur on the sides of the region of
$\beta^s$.  Let $Q_i,\, Q_p,\, T_i,\, T_p$ denote the number of inner
quadrilaterals, peripheral quadrilaterals, inner triangles, and
peripheral triangles, respectively.  For example, in Figure
\ref{fig:s-decomp}\,(b), $Q_i=4,\, Q_p=0,\, T_i=4,\, T_p=6$.

\begin{lemma}
\label{lemma:braid-quads}
  For any positive root $\beta\in B_r$, the region inside the curve $\gamma$ of $\beta^s$ contains 
  \begin{itemize}
\item $\ Q_i = (r-3)(s-1)\ $ inner quadrilaterals, 
\item $\ Q_p \leq (r-2)\ $ peripheral quadrilaterals, and 
\item $\ T_p + Q_p \leq 2(r-2)\ $ peripheral faces that are not bigons.
\end{itemize}
If $\beta=\delta_r$ or $\bar\delta_r$, there are no peripheral quadrilaterals and $T_p = 2(r-2)$ peripheral triangles.
\end{lemma}

\begin{proof}
Let $T=T_i+T_p$ and $Q=Q_i+Q_p$, and let $B$ be the number of bigons.
The triangles, quadrilaterals, and bigons of $T^2\times \{0\}$ are
nearly in one-to-one correspondence with triangles, quadrilaterals,
and bigons of the projection graph of $\beta^s$, except at the start
and end of the braid.  We make the correspondence complete by pulling
all strands at the start of the braid projection graph into a single
vertex, and all strands at the end of the braid projection graph into
a single vertex.  The result is a graph on a 2--disk $D$, decomposed
into quadrilaterals, triangles, and bigons. See Figure \ref{fig:s-decomp}\,(c).

Let $c=s(r-1)$ denote the number of crossings of $\beta^{s}$.  Let
$v=2+c$ denote the number of vertices on $D$.  Let $e=2c+r$ be the
number of edges, and $f=B+T+Q$ be the number of faces.  Now,
$\chi(D)=v-e+f=1$ implies $B+T+Q = c+r-1$.  Moreover, $2e=4c+2r =
2B+3T+4Q +2(s+1)$.  We now subtract these equations:
\begin{align*}
2B+3T+4Q &= 4c+2r-2s-2 \\
3B+3T+3Q &= 3c+3r-3 \\
\hline 
\strut -B +Q &= c-r-2s+1 . 
\end{align*}
Using the formula $c=s(r-1)$, the last equation simplifies to
\begin{equation}\label{eq:quad-bigon}
  Q = (B-2) + (r-3)(s-1).
\end{equation}
We claim that $Q_p = B-2$, which will be proved using the normal form for $\beta=
\sigma_{i_1}\sigma_{i_2}\dots\sigma_{i_{n-1}}$, particularly Lemma
\ref{lemma:normal-form}.

Note we obtain a bigon at the start of the braid $\beta^s$, between the $j$-th and $(j+1)$-st
strands, $1<j<n-1$, if and only if in the word of $\beta$, $\sigma_j$
appears before both $\sigma_{j-1}$ and $\sigma_{j+1}$.  For the $n$-th
and $(n+1)$-st strands, there will be a bigon if and only if
$\sigma_{n-1}$ appears before $\sigma_{n-2}$.  Finally, there will be
a bigon between the first two strands if and only if $\sigma_1$
appears first in the word $\beta$.  By Lemma \ref{lemma:normal-form},
we pick up a bigon on top for the first chain in the indices of
$\beta$, in equation \eqref{eq:normal-braid}, and one for each additional
chain of length at least two. 

Next, we obtain a quadrilateral at the start of $\beta^s$, between the
$j$-th and $(j+1)$-st strands, $1<j<n-1$, if and only if $\sigma_j$
appears after both $\sigma_{j-1}$ and $\sigma_{j+1}$ in the word of
$\beta$.  There can be no quadrilaterals in the first or last strand
positions.  By Lemma \ref{lemma:normal-form}, we pick up a
quadrilateral for each chain in the indices of $\beta$ which has
length at least two, except the first. In particular, the peripheral
quadrilaterals at the start of the braid are in bijection with all but
one of the bigons at the start of the braid.  A similar analysis
applies at the end of the braid.

Therefore, we conclude that the peripheral quadrilaterals are in
bijection with all but two of the (peripheral) bigons, so that $Q_p =
B-2$.  It follows from equation (\ref{eq:quad-bigon}) that $Q_i = (r-3)(s-1)$.
Moreover, there are $r-1$ faces at the start of the braid,
and $r-1$ faces at the end of the braid, so $Q_p+T_p+B = 2(r-1)$.  Thus,
$$Q_p + B \leq 2r-2 \quad \Rightarrow \quad Q_p + (Q_p +2) \leq 2r-2
\quad \Rightarrow \quad Q_p \leq r-2.$$ 
Finally, at least 2 peripheral faces are bigons, so $Q_p + T_p \leq
2r-4$.  When $\beta=\delta_r$ or $\bar\delta_r$, there is only one
chain in equation \eqref{eq:normal-braid}, hence $B=2,\, Q_p = 0$ and
$T_p = 2(r-2)$.
\end{proof}

Recall that the root $\beta$ of $\Delta^2$ is a positive braid. 
Thus every inner quadrilateral, as in Lemma \ref{lemma:braid-quads},
is formed by two parallel over--strands crossing two parallel
under--strands.  At every inner quadrilateral, we insert a
tetrahedron, called a {\em medial tetrahedron}.  See
Figure \ref{fig:inner_tet}, left.  Two faces of every medial
tetrahedron can be seen  from $T^2 \times \{1\}$
(``top''), and the remaining two faces can be seen from $T^2 \times \{-1\}$ (``bottom'').

We also insert a medial tetrahedron at every peripheral quadrilateral
of the braid $\beta^s$. These look exactly the same as the tetrahedron
in Figure \ref{fig:inner_tet}, left, except with one corner
truncated. Again, two faces can be seen from the top, and two from
the bottom.

A consequence of inserting these medial tetrahedra is that when we
view the region of $\beta^s$ from above (or below), all faces are
triangles or bigons. Every bigon will collapse to an ideal edge, and
almost all of the triangles can be glued in pairs:

\begin{figure}
	\includegraphics[height=0.5in]{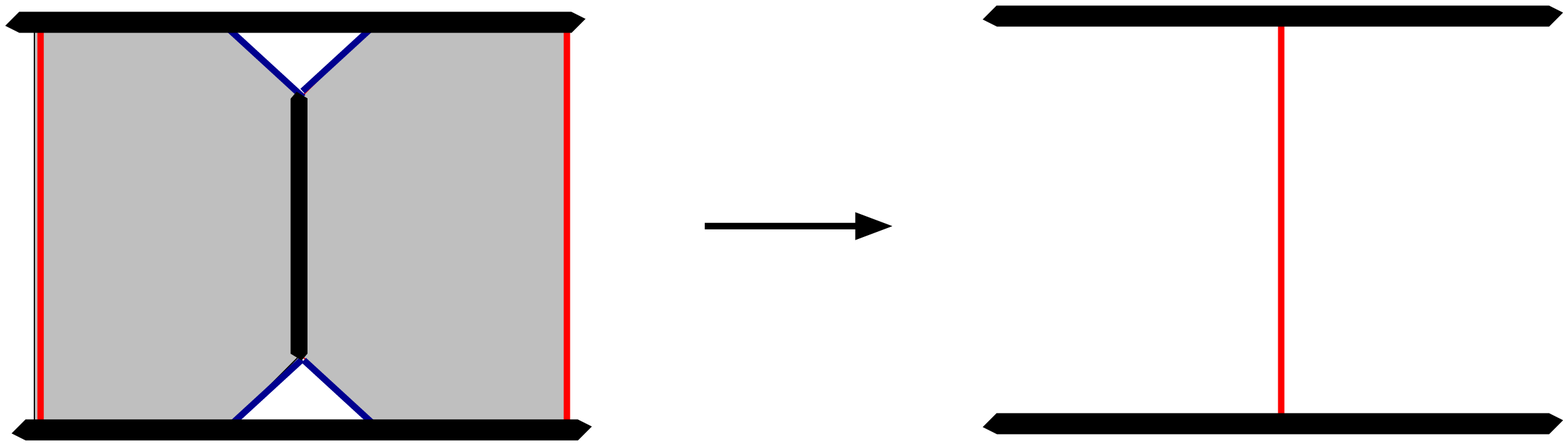}
 \hspace{0.7in}
	\includegraphics[height=0.5in]{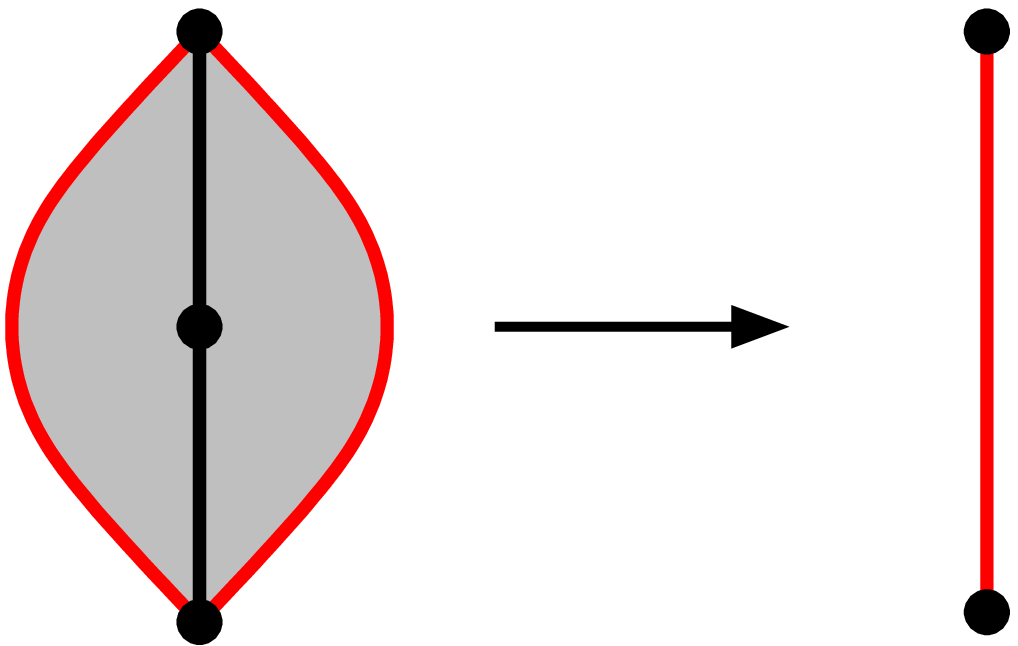}
 
	\caption{Collapsing two triangles with two common edges to an
    edge.\qquad  Left: Shown on $T^2\times 0$ with edges of $\beta^s$. Right: Shown with ideal vertices.}
\label{fig:sailboat}
\end{figure}

\begin{lemma}
\label{lemma:collapse}
Inside the region of $\beta^s$, let $t$ be a triangular face that is
visible from the top and contains no peripheral edges, possibly a
non-peripheral top face of a medial tetrahedron.
Then $t$ shares two edges with an adjacent triangle $t'$. The third
edges of $t$ and $t'$ are isotopic in $M(n,m,r,s)$. Thus $t$ can be
glued to $t'$, identifying their third edges. When this gluing
operation is viewed from $T^2 \times \{1\}$, the two triangles
collapse to a single edge. See Figure \ref{fig:sailboat}.

A similar statement holds for a triangular face visible from $T^2 \times \{- 1\}$. 
\end{lemma}

\begin{proof}
If $t$ is not peripheral, then all three sides of $t$ are inside the
region of $\beta^s$ determined by $\gamma$.  Because $\beta^s$ is a
positive braid, two ideal edges of $t$ connect to a strand of the
projection diagram between consecutive under-crossings. These two
edges are shared with an adjacent triangle $t'$, as in Figure \ref{fig:sailboat}.

Now, observe that every edge of $t$ is isotopic in $M(n,m,r,s)$ to a
corresponding edge of $t'$. Therefore, we may glue these two triangles
together, by folding them toward the top. After the gluing, all that
is visible from the top is a single ideal edge.
\end{proof}

\begin{figure}
	\includegraphics[height=1.25in]{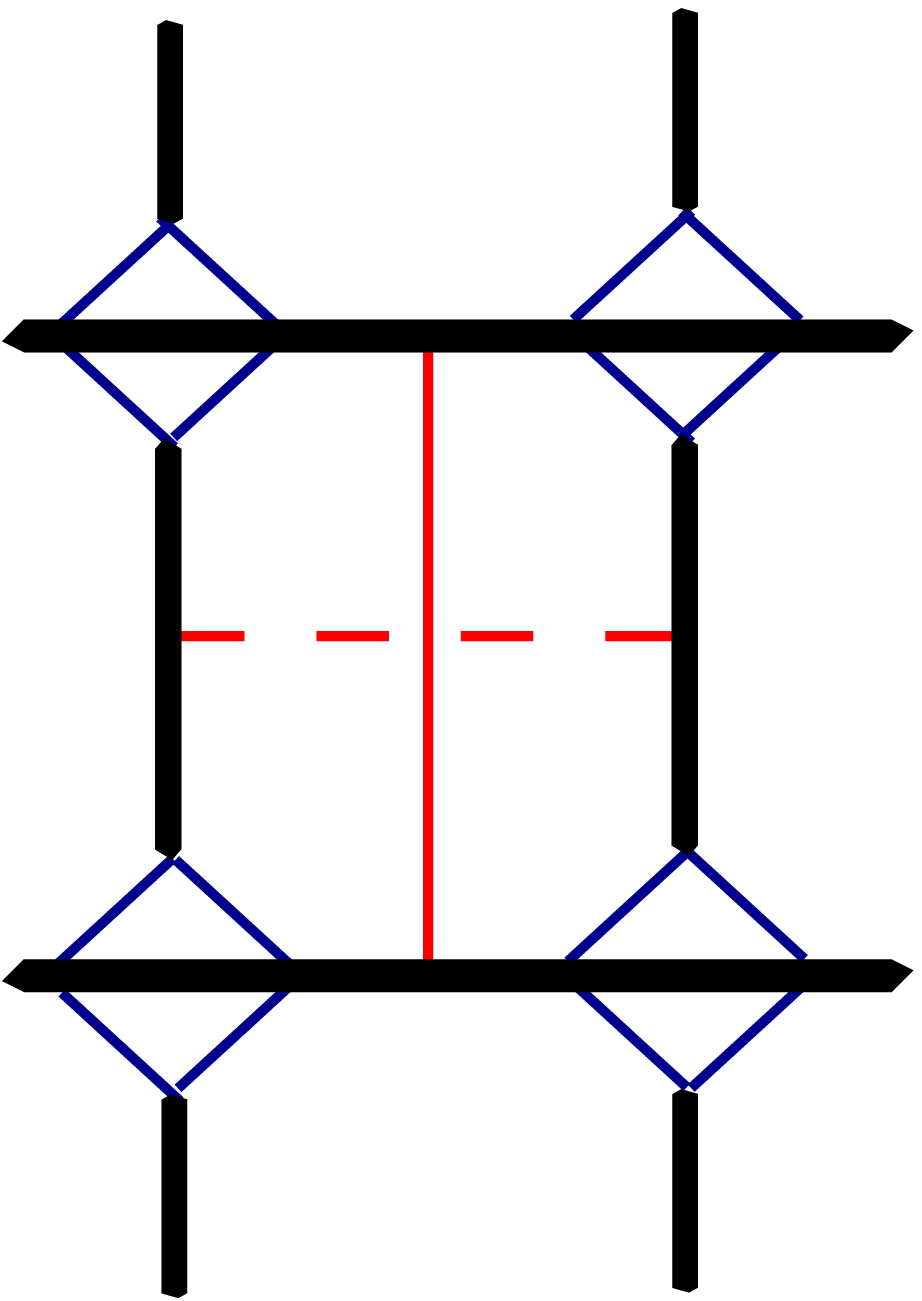}
 \hspace{0.7in}
	\includegraphics[height=1.3in]{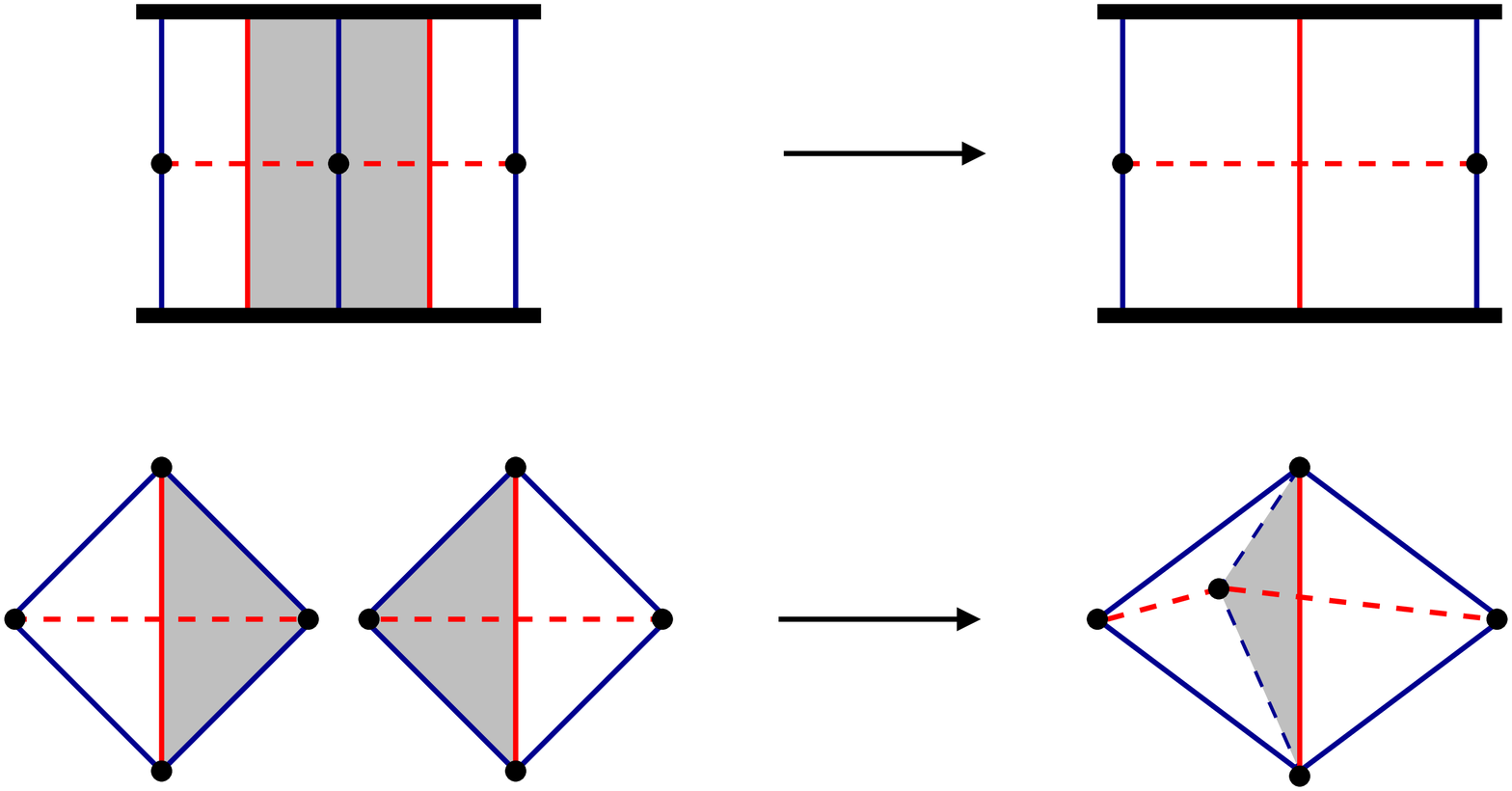}
	\caption{Left: Medial tetrahedron inserted at every inner
          quadrilateral. Right: Adjacent medial tetrahedra are glued
          together when their triangular faces with common edges are collapsed.}
\label{fig:inner_tet}
\end{figure}

See Figure \ref{fig:inner_tet}, right, for an illustration of this gluing, in the case where $t$ and $t'$ both belong to medial tetrahedra.

We now return to the proof of Lemma \ref{lemma:tetrbound2}.  As in the proof
of Lemma \ref{lemma:tetrbound1}, outside both the flattened
half--disks and the curve $\gamma$ encircling $\beta^s$, regions on
$T^2\times\{0\}$ are either bigons, in which case they collapse, or
meet one or both of the intersections of the link $L$ with $T^2 \times
\{0\}$.  In this case, the regions also meet the $s$ edges on either
side of the $s$ overpasses.  There are either two $(s+4)$--gons, or a
single $(2s+6)$--gon.  One $(s+4)$--gon is shown as a shaded region in
Figure \ref{fig:s-decomp}\,(a).

\begin{figure}
  \begin{tabular}{c}
    \includegraphics[height=1.4in]{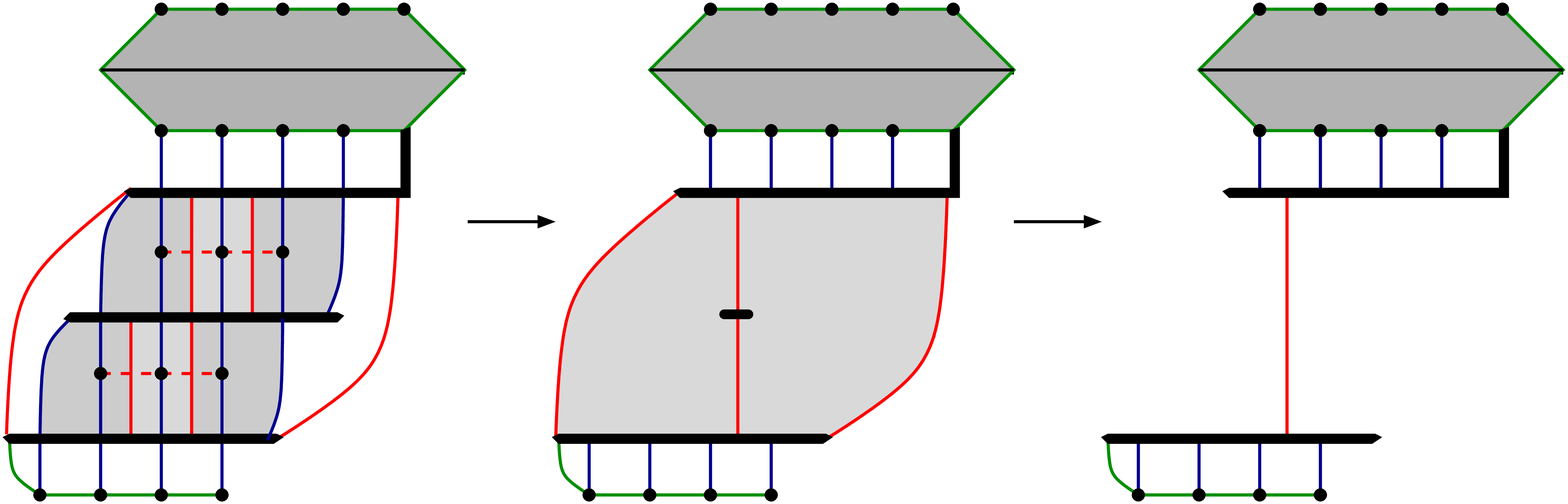}\\
\\ 
 \includegraphics[height=1.25in]{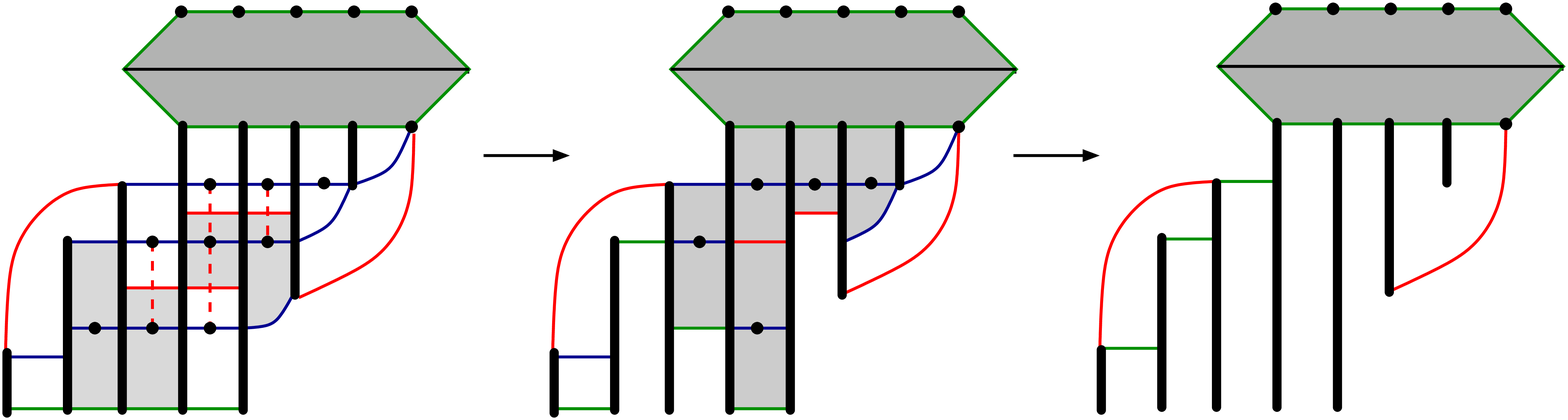} \\
 
  \end{tabular}
  \caption{For a twisted torus braid with $s\neq 0$ and
    $\beta=\bar\delta_r$, attach medial tetrahedra as in Figure
    \ref{fig:sailboat} right, and repeatedly collapse using Lemma
    \ref{lemma:collapse}.  In the top row, $T^2 \times \{0\}$ is seen
    from the top, and in the bottom row from the bottom. Thick black
    lines indicate ideal vertices. }
\label{fig:ttk_collapse}
\end{figure}

{\bf Case: $\beta=\delta_r$ or $\bar\delta_r$.}\quad We now add
$2(s-2)$ edges on $T^2\times\{0\}$, as shown in Figure
\ref{fig:ttk_collapse}, subdividing the two $(s+4)$--gons or single
$(2s+6)$--gon into two hexagons or a single decagon, respectively, as
well as $2(s-2)$ triangles adjacent to the braid.  Viewed from the
top, all inner faces can be glued in pairs to collapse to edges, and
then triangles adjacent to the braid can also be glued in pairs, to
collapse to a single edge.  This is shown in the top row of Figure
\ref{fig:ttk_collapse}.  Thus, from the top, all that is left of the
two $(s+4)$--gons or single $(2s+6)$--gon are two hexagons or a single
decagon, respectively, as in Figure \ref{fig:overpasses2}.  Viewed
from the bottom, after the inner quadrilaterals are collapsed, the
$2(s-2)$ new triangles remain, as well as the two hexagons or single
decagon.  When we cone to top and bottom, pyramids over the two
hexagons or single decagon will be glued along those faces, and so we
may perform stellar subdivision.  Lemmas \ref{lemma:subdivide} and
\ref{lemma:braid-quads} apply, and we count:
\begin{itemize}
\item $2(r+1)$ tetrahedra from each pair of half--disks, which are $(r+1)$--gons.
\item $12$ tetrahedra from the two hexagons, or $10$ tetrahedra from the single decagon.
\item $(r-3)(s-1)$ medial tetrahedra.
\item $2(r-2)$ tetrahedra from coning peripheral triangles to the top.
\item $2(s-2)$ tetrahedra from coning other triangles to the bottom.
\end{itemize}
In this case, all the $(r-3)(s-1)$ medial tetrahedra are incident to
the single collapsed edge seen from the top (the edge shown in the top
right of Figure \ref{fig:ttk_collapse} and middle of Figure
\ref{fig:overpasses2}(a)\,).

\begin{figure}
  \begin{tabular}{cc}
	\includegraphics[height=1.25in]{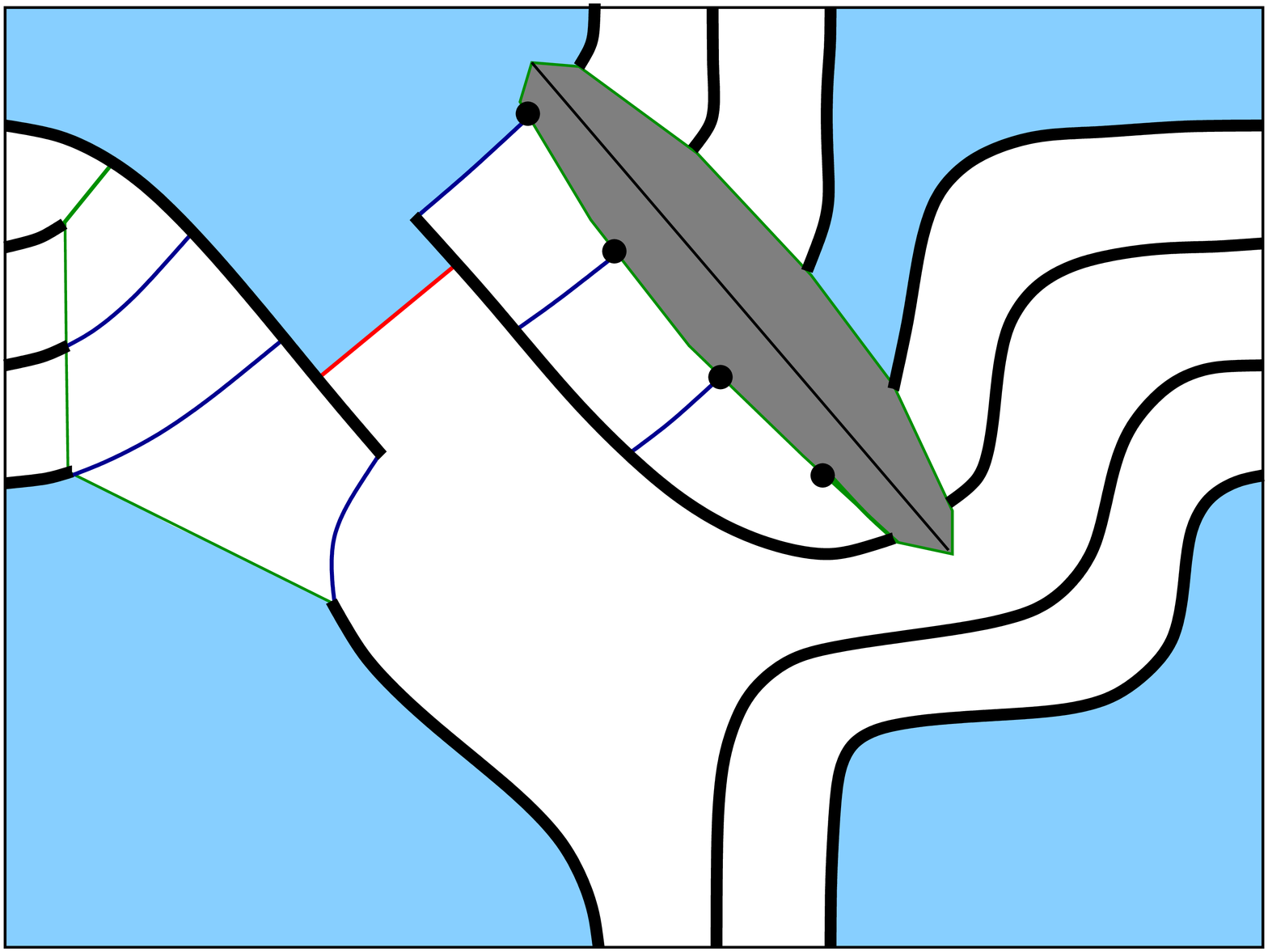} \ \ &
\includegraphics[height=1.25in]{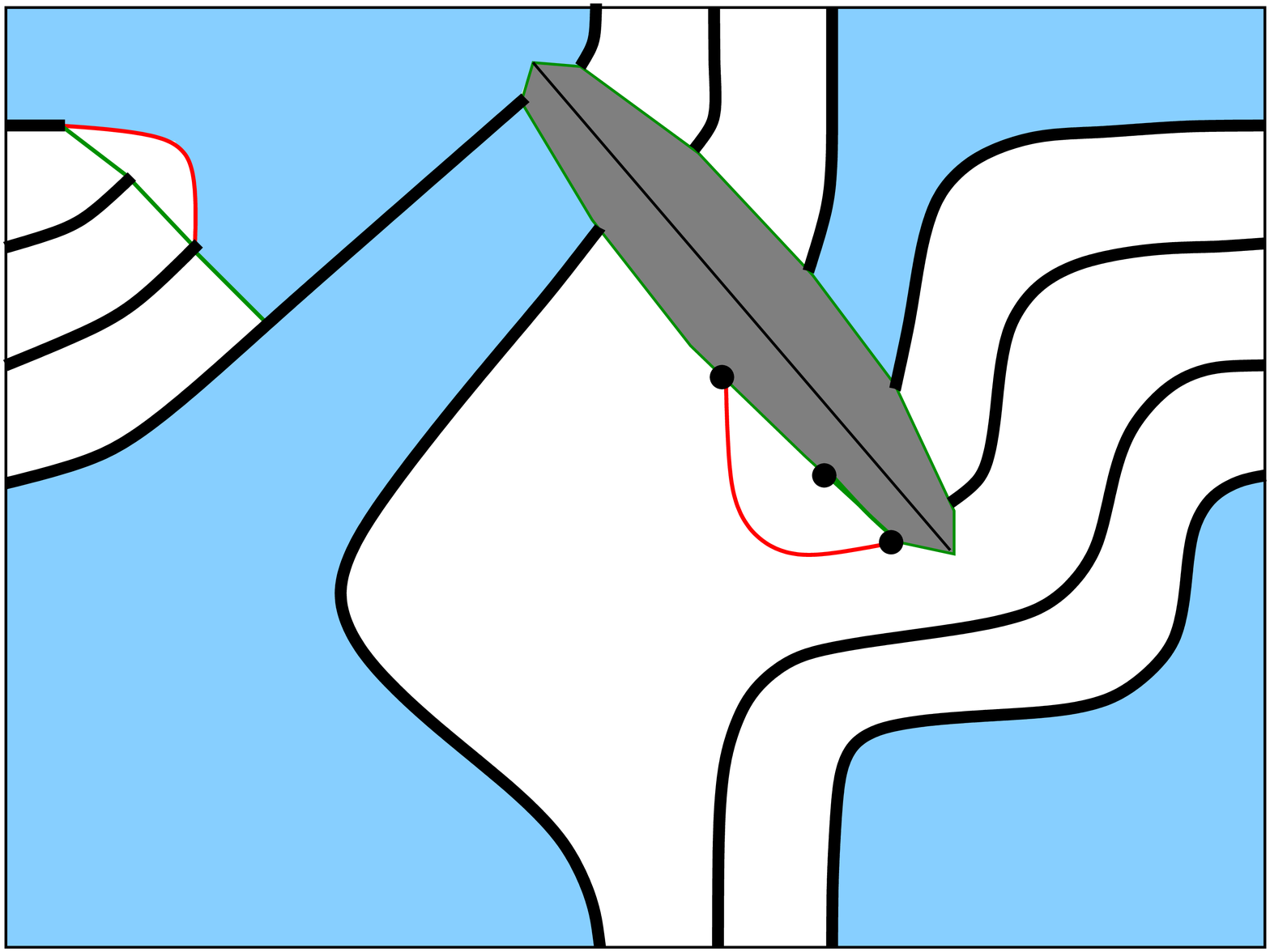} \\
    (a) & (b) 
  \end{tabular}
	\caption{After collapsing as in Figure \ref{fig:ttk_collapse},
          regions on $T^2 \times \{0\}$ seen from the top in (a) and from
          the bottom in (b).}
\label{fig:overpasses2}
\end{figure}

{\bf Case: $\beta$ is any positive root.}\quad 
For simplicity (and unlike the previous case), we will not attempt to
collapse any additional triangles, beyond what was done in Lemma
\ref{lemma:collapse}.  Instead, the two $(s+4)$--gons or the single
$(2s+6)$--gon will simply be coned to $T^2 \times \{ \pm 1\}$, and
subdivided into tetrahedra by stellar subdivision as in Lemma
\ref{lemma:subdivide}.  Also, by Lemma \ref{lemma:collapse}, for each
peripheral quadrilateral, only one triangular face of a medial
tetrahedron will be visible from the top, and one from the bottom,
after collapsing.  Thus, peripheral triangles and quadrilaterals each
contribute two tetrahedra by coning one triangle to the top, and one
to the bottom.  We count:
\begin{itemize}
\item $2(r+1)$ tetrahedra from each pair of half--disks, which are $(r+1)$--gons.
\item $2(s+4)$ tetrahedra from stellar subdivision, assuming worst case of two $(s+4)$--gons.
\item $Q_i = (r-3)(s-1)$ medial tetrahedra from inner quadrilaterals.
\item $Q_p \leq (r-2)$ medial tetrahedra from peripheral quadrilaterals.
\item $2(Q_p + T_p) \leq 4(r-2)$ tetrahedra from coning peripheral faces to the top and bottom.
\end{itemize}
Adding these counts together completes the proof of Lemma \ref{lemma:tetrbound2}.  
\end{proof}

We can now prove the upper volume bounds of Theorem \ref{thm:s-bound}.

\begin{proof}[Proof of Theorem \ref{thm:s-bound}]
The link complement $S^3 \setminus T(p,q,r,s)$ is obtained by Dehn
filling the manifold $M(p,q,r,s)$, so $\vol(T(p,q,r,s))$ is bounded
above by the volume of $M(p,q,r,s)$ \cite{Thurston}.  By Lemma
\ref{lemma:Mpqr-homeo}, the volume of $M(p,q,r,s)$ is the same as that
of $M(m,n,r,s')$ or $M(n,m,r,s')$, where $s' = s \mod r$, $0\leq s' < r$ and $n$, $m$
are as in the statement of that lemma.  

The volume of any ideal tetrahedron is at most $v_3$.  If $r=2$ and
$s=0$, Lemma \ref{lemma:tetrbound1} implies that $M(1,1,2,0)$ can be
decomposed into $10$ ideal tetrahedra, so $\vol(M(1,1,2,0)) \leq
10v_3$.  Notice that $M(1,1,2,1)$ differs from $M(1,1,2,0)$ by a
single half--twist in a 2--punctured disk.  Hence these two manifolds
have the same volume \cite{adams:3-punct}.  Thus $\vol(T(p,q,2,s)) <
\vol(M(1,1,2,0)) \leq 10v_3$.

If $r>2$ and $s=0 \mod r$, Lemma \ref{lemma:tetrbound1} implies the
manifold $M(n,m,r,0)$ (or $M(m,n,r,0)$) can be decomposed into
$2(r+4)$ tetrahedra, or $2(r+5)$ tetrahedra, depending on whether
$n+m=r$ or not, respectively.  

Finally if $r>2$ and $s\neq 0 \mod r$, then Lemma
\ref{lemma:tetrbound2} applies, and the manifold can be decomposed
into at most $rs'+3r-s'+9 = (r-1)s' + 3r + 9$ tetrahedra if
$\beta=\delta_r$, and $rs'+6r-s'+3= (r-1)s' + 6r+3$ otherwise.  Since
$0<s'<r$, we obtain volume bounds $v_3(r^2 + r + 10)$ and $v_3(r^2 +
4r + 4)$, respectively.
\end{proof}

When $r=2$, the bound of Theorem \ref{thm:s-bound} is sharp.  In the
special case when $s=0 \mod r$ and $m+n=r$, the above proof gives the
better bound $v_3(2r+8)$.

\subsection{Links with volume approaching $10 v_3$.}
The following construction gives an explicit family of twisted torus knots with $r=2$ whose volumes approach $10v_3$. This will prove Theorem \ref{thm:sharp2}, and demonstrate the sharpness of Theorem \ref{thm:s-bound} when $r=2$.

\begin{proof}[Proof of Theorem \ref{thm:sharp2}]
The manifold $S^3\setminus T(p,q,2,2N)$ is obtained from $M(p,q,2,2N)$
by Dehn filling.  Lemma \ref{lemma:Mpqr-homeo} implies that
$M(p,q,2,2N)$ is homeomorphic to $M(1,1,2,0)$.  We begin the proof by showing 
that $M(1,1,2,0)$ is a hyperbolic manifold obtained by gluing $10$
regular ideal tetrahedra, hence has volume exactly $10v_3$.

Notice that the manifold $M(1,1,2,0)$ has a $\Z^2$ cover by the
infinite chain--link--fence complement, Figure \ref{fig:chain-link}.
The infinite chain--link--fence complement is studied in detail by
Agol and Thurston in \cite[Appendix]{lackenby:volume-alt}.  In
particular, they find a subdivision of this link complement into
regular ideal tetrahedra.

\begin{figure}
  \includegraphics[height=1.5in]{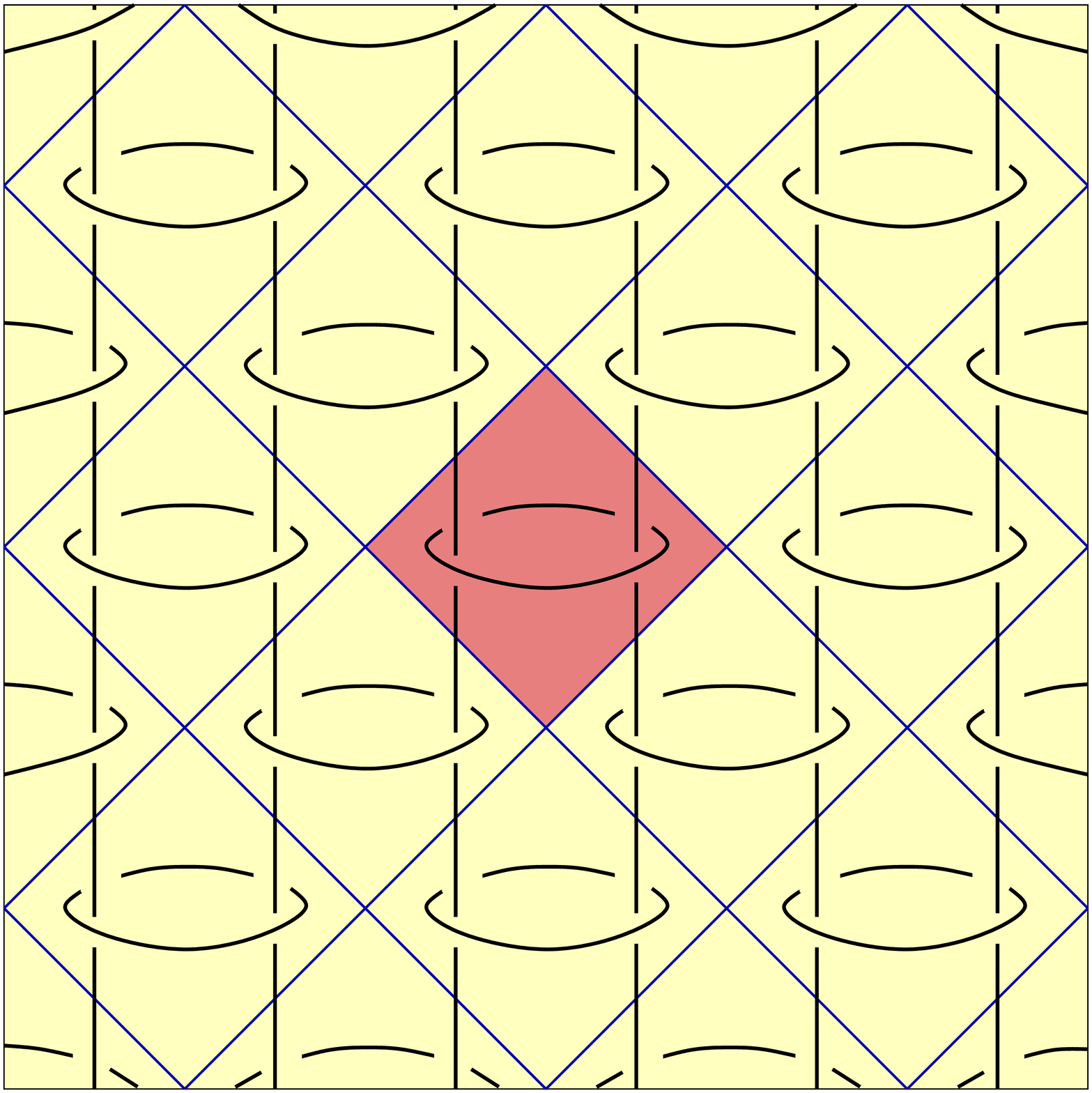}
  \caption{$M(1,1,2,0)$, shown in center, is covered by the infinite
    chain--link--fence complement.}
  \label{fig:chain-link}
\end{figure}

This subdivision is obtained first by slicing the chain--link--fence
complement in half along the projection plane, then slicing up
half--disks bounded by crossing circles as in our proof of Lemma
\ref{lemma:tetrbound1} above, and opening them up and flattening them.
These are coned to points above or below the projection plane,
yielding four tetrahedra per  crossing disk.  The remainder of the
chain--link--fence complement consists of pyramids over regular
hexagons.  This is illustrated in Figure 17 of
\cite{lackenby:volume-alt}, where shaded triangles come from disks
bounded by circles.  To finish the decomposition into tetrahedra, Agol
and Thurston replace two pyramids glued over a hexagon face with a
stellar subdivision into six tetrahedra, as in Lemma
\ref{lemma:subdivide}.

All these tetrahedra in the decomposition of the chain--link--fence
complement are now seen by a circle packing argument to be regular
ideal tetrahedra.  Notice that the subdivision is invariant under the
action of $\Z^2$ corresponding to our covering transformation.  Thus
the regular ideal tetrahedra descend to give a decomposition of
$M(1,1,2,0)$ into ideal tetrahedra.  Tracing through the proof of
Lemma \ref{lemma:tetrbound1}, we see that these ideal tetrahedra
agree with those of our decomposition.  Since there are $10$ such
tetrahedra, the volume of $M(1,1,2,0)$ is $10v_3$.

Another way to see that all tetrahedra are regular follows from the
fact that all the edges of this triangulation are 6--valent. In this
case, the ideal tetrahedra satisfying the gluing equations have all
dihedral angles $\pi/3$, so they are regular ideal tetrahedra.  Since
all links of tetrahedra are equilateral triangles, they are all
similar, and all edges of any triangle are scaled by the same factor
under dilations.  Hence, the holonomy for every loop in the cusp has
to expand and contract by the same factor (i.e., it is scaled by
unity), so it is a Euclidean isometry. This implies that the regular
ideal tetrahedra are also a solution to the completeness equations.

Finally, recall that every knot $T(p,q,2,2N)$ is obtained by Dehn filling three
of the four boundary tori of $M(1,1,2,0)$. Two of these tori correspond 
to the components of the Hopf link, or equivalently the
top and bottom boundary components of $T^2\times I$ in $M(1,1,2,0)$, 
while the third is the crossing circle 
encircling the two strands of the knot
of slope $(1,1)$ on $T^2 \times \{0\}$.

Now, choose a pair of (large) integers $p$ and $q$, such that $\gcd(p, q) = 1$. 
In other words, there exist integers $(u,v)$, such that $uq-pv=1$.  We may embed
$T^2\times I$ into the complement of the Hopf link via the matrix
$$A = \begin{bmatrix} u & p-u \\ v & q-v \end{bmatrix}.
\qquad \mbox{Note that} \quad \begin{bmatrix} u & p-u \\ v & q-v \end{bmatrix}
 \begin{bmatrix} 1 \\ 1 \end{bmatrix} =  \begin{bmatrix} p \\ q \end{bmatrix}, $$
hence this embedding sends the curve of slope $(1,1)$ on $T^2$ to a $(p,q)$ torus knot in $S^3$.

Consider the Dehn filling slopes in this construction.
We will perform $(1,N)$ Dehn filling on the crossing circle in $T^2\times I$, thereby inserting $2N$ crossings between a pair of strands in the $(p,q)$ torus knot. As for the top and bottom tori of $T^2\times I$, we will fill them along the slopes that become meridians of the Hopf link after embedding via the matrix $A$. In other words, in the original framing on $T^2 \times I$, the top and bottom Dehn filling slopes are 
$$A^{-1}  \begin{bmatrix} 1 \\ 0 \end{bmatrix} = \begin{bmatrix} q-v \\ -v \end{bmatrix} 
\quad \mbox{and} \quad
A^{-1}  \begin{bmatrix} 0 \\ 1 \end{bmatrix} = \begin{bmatrix} u-p \\ u \end{bmatrix}. $$

Thus, so long as $N$ is large and $p=p_N$ and $q=q_N$ are also large, the Dehn filling slopes will be long. As $(p_N, q_N) \to (\infty, \infty)$, the length of the slopes approaches $\infty$. Thus the volume of the Dehn filled manifold, $S^3 \setminus T(p_N, q_N, 2, 2N)$, will approach $\vol(M(1,1,2,0)) =
10v_3$.
\end{proof}

\subsection{$T$--links}

Let $K$ be $T((p,q),(r_{1},s_{1}),\dots ,(r_k,s_k))$.  Let $C_1\cup
C_2$ be the Hopf link as above.  For $i=1,\ldots, k$, augment the link
$K$ with unknots $L_i$ that
encircle the $r_i$ strands of the $i$-th braid $\beta_i$.  Let $s_i'=s_i\mod
r_i$, so that $0\leq s_i' < r_i$.  Let
$$M(p,q,r_1,s_1',\ldots,r_k,s_k') = S^3\setminus \left(C_1 \cup C_2
\cup (\cup_{i=1}^k L_i)\cup K\right).$$
Note $S^3\setminus K$ is
homeomorphic to a Dehn filling on $M(p,q,r_1,s_1',\ldots,r_k,s_k')$.

\begin{lemma}
  Suppose $s_i'=0$, $i=1, \dots, k$.  Then $M(p,q,r_1,0,\dots,r_k,0)$
  can be decomposed into at most $r_1^2+9r_1-8$ ideal tetrahedra.
\label{lemma:tetrbound3}
\end{lemma}

\begin{proof}
As above, we will decompose $M(p,q,r_1,0,\ldots ,r_k,0)$ into
tetrahedra by first cutting the manifold into two pieces along the
torus $T^2 \times \{0\}$.  The common boundary of these two pieces is
shown in Figure \ref{fig:Pk-decomp}.  After cutting along $T^2\times
\{0\}$, the punctured disks $D_i$ bounded by $L_i$ are cut into two,
sliced and flattened onto $T^2 \times \{ 0 \}$.  Each half--disk $D_i$
can be divided into two $(r_i+1)$-gons.  We obtain $2(r_i+1)$
tetrahedra from each of these, as above, by first coning to the
boundary $T^2\times \{1\}$ or $T^2\times \{-1\}$, obtaining two
pyramids for each half--disk, and then applying Lemma
\ref{lemma:subdivide}.  (If some $r_i=2$, we can improve this bound,
but we won't use this fact.)  This gives a total of $\sum_{i=1}^k
2(r_i+1)$ tetrahedra from half--disks.

\begin{figure}
	\includegraphics[height=1.3in]{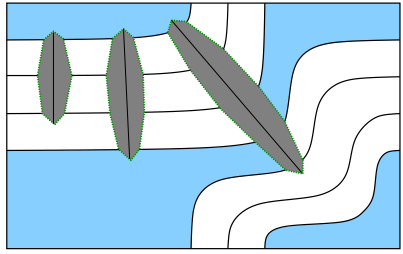}
	\caption{T-link with all $s_i'=0$, and all half--disks
          flattened.  $U$ is the shaded region.  A polygon
          remains after collapsing bigons.}
\label{fig:Pk-decomp}
\end{figure}

Now we consider regions other than half--disks.  Each such region is
either a bigon, in which case it contributes no tetrahedra, or meets
at least two edges on the end of some $D_i$.  In particular, one
region $U$ (the ``top left corner'' of Figure \ref{fig:Pk-decomp},
shown shaded) will meet $2k$ such edges, two for the end of each $D_i$
meeting in that region, as well as additional edges as the region
connects to other regions by identifications on the torus.

The total number of tetrahedra for all these regions will be as large
as possible when each region meets as few edges of the $D_i$ as
possible.  This will happen when $U$ meets only the $2k$ edges
corresponding to side edges of each $D_i$, and then just two more
edges, either both from $D_1$ or one from $D_1$ and one from $D_j$, to close off.  Additionally, the other
end of each $D_i$ will be in a region meeting no other end of another
disk 
and this region will meet exactly two edges from the end of
$D_i$ and exactly two other edges of some other disks
to close off.  Hence each such region has four edges total.  Thus when we have
the maximum number of tetrahedra possible, we will have one region
with $2k+2$ edges, and $k$ quadrilaterals.

As before, cone these to pyramids lying above and below $T^2\times
\{0\}$.  Lemma \ref{lemma:subdivide} implies this can be divided
into at most $(2k+2) + 4k = 6k+2$ tetrahedra.

Since the $r_i$'s are strictly decreasing and $r_k \geq 2$, it follows
that $k \leq r_1 - 1$ and $\sum_{i=1}^k r_i \leq \sum_{i=2}^{r_1} i =
(r_1^2+r_1-2)/2$.  Hence the number of tetrahedra is bounded by:
\begin{eqnarray*}
\sum_{i=1}^k 2(r_i+1) + (6k+2) = (6k+2) + 2k + 2 \sum_{i=1}^k r_i \\
\leq \ 8(r_1-1) + 2 + (r_1^2+r_1-2) \ = \ r_1^2 + 9r_1 -8
\end{eqnarray*}
This completes the proof of Lemma \ref{lemma:tetrbound3}.
\end{proof}

\begin{lemma}
Suppose $s_i'\neq 0$, for some $i$, $1\leq i \leq k$.  The manifold $M(p,q,(r_1,s_1'), \dots, (r_k, s_k'))$ can be
  decomposed into $t$ ideal tetrahedra, where $t$ is at most
  $ \tfrac{1}{3}r_1^3 + \tfrac{5}{2} r_1^2 + 5 r_1 - 5 $.
  \label{lemma:tetrbound4}
\end{lemma}

\begin{proof}
See Figure \ref{fig:tlink1}, which generalizes both Figure
\ref{fig:s-decomp} and Figure \ref{fig:Pk-decomp}.
First assume that $s_i'\neq 0,\forall \  i=1,\ldots,k$.
Each region with $\beta_i^{s_i'}$, such as the
$s_i'$ overpasses in Figure \ref{fig:tlink1}, can be subdivided as in
Lemma \ref{lemma:braid-quads}.  Thus, we can repeatedly apply the
methods of Lemma \ref{lemma:tetrbound2} in the general case that
$\{\beta_i\}$ are any positive roots.  We count:

\begin{figure}
	\includegraphics[height=1.5in]{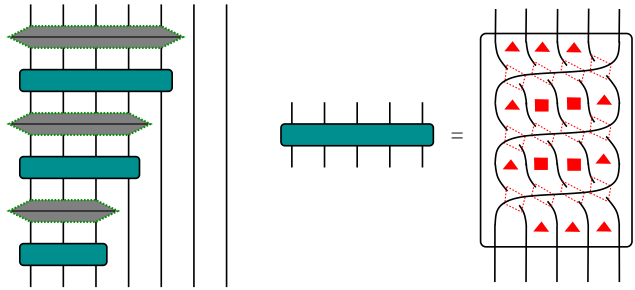}
	\caption{Schematic figure of a general T-link, with triangles
          and quadrilaterals indicated for the root $\bar\delta_5$.}
\label{fig:tlink1}
\end{figure}

\begin{itemize}
\item $\sum_{i=1}^k2(r_i+1)$ tetrahedra from each pair of half--disks $D_i$.
\item $\sum_{i=1}^k  \left[ (r_i-3)(s_i'-1) + (r_i-2) \right]$ medial tetrahedra.
\item $\sum_{i=1}^k4(r_i -2)$ tetrahedra from coning remaining at most $2(r_i -2 )$ peripheral triangles to the top and bottom.
\end{itemize}

To count the triangles in the unbounded regions in Figure
\ref{fig:tlink1}, we recall from the proof of Lemma
\ref{lemma:tetrbound3} that all regions outside the half--disks in
Figure \ref{fig:Pk-decomp} together contribute at most $6k+2$
tetrahedra.  However, in Figure \ref{fig:tlink1}, there are
$2\sum_{i=1}^k s_i'$ additional side edges from the regions with
overpasses.  Because these side edges are adjacent, the number of
unbounded regions in Figure \ref{fig:tlink1} is the same as in the
case where all $s_i'=0$.  Since we subdivide these regions into
tetrahedra, one for each edge, it follows that the number of
tetrahedra from these regions is at most the previous count plus the
number of additional side edges: $6k+2 + 2\sum_{i=1}^k s_i'$.
(Whenever $\beta_i=\delta_{r_i}$ or $\bar\delta_{r_i}$, this triangulation can be
improved by the methods of Lemma \ref{lemma:tetrbound2}.)

Therefore, using $s_i' \leq r_i-1$ and $k\leq r_1-1$, the total number
of tetrahedra is bounded by:
\begin{eqnarray*}
t & \leq & \sum_{i=1}^k \big( 2(r_i+1) + (r_i-3)(s_i'-1) + 5(r_i -2) + 2s_i'\big) + 6k+2 \\
& \leq & \sum_{i=1}^k \big(2(r_i+1) + (r_i-3)(r_i-2) + 5(r_i-2) + 2(r_i-1)\big) + 6k+2 \\
& = & \sum_{i=1}^k \big(r_i^2 + 4r_i\big) +2k+2
\ \leq\ \sum_{i=2}^{r_1} \big(i^2 + 4i \big) +2 r_1  \\
& = & \tfrac{1}{3}r_1^3 + \tfrac{5}{2}r_1^2 + \tfrac{25}{6}r_1 -5
\ < \ 
\tfrac{1}{3} r_1^3 + \tfrac{5}{2} r_1^2 + 5 r_1 -5.
\end{eqnarray*}

If $s_i'=0$ for some $i$, $1\leq i \leq k$, this region does not
contribute to the count of medial tetrahedra or tetrahedra from coning
peripheral triangles.  Hence, $t$ is bounded as above.
\end{proof}

\begin{proof}[Proof of Theorem \ref{thm:T-bound}]
Let $L$ be the following $T$--link
$$L = T((p,q),(r_1, s_1,\beta_1), \dots,
(r_k,s_k,\beta_k)).$$  If $s_i = 0 \mod r_i$ for all $i$, then
$S^3\setminus L$ is obtained by Dehn filling $M=M(p,q,r_1, 0, \dots,
r_k,0)$, and by Lemma \ref{lemma:tetrbound3}, $M$ can be decomposed
into at most $r_1^2+9r_1-8$ tetrahedra.  Hence,
$$ \vol(L) < v_3(r_1^2+9r_1-8). $$

If some $s_i \neq 0 \mod r_i$, then by Lemma \ref{lemma:tetrbound4},
$S^3\setminus L$ is obtained by Dehn filling a manifold which
decomposes into at most
$\tfrac{1}{3} r_1^3 + \tfrac{5}{2} r_1^2 + 5 r_1 -5$ tetrahedra.  Thus,
$$ \vol(L) <  v_3 \left( \tfrac{1}{3} r_1^3 + \tfrac{5}{2} r_1^2 + 5  r_1 -5 \right).$$

\vspace{-2ex}
\end{proof}


\section{Twisted torus knots with large volume}\label{sec:lower-bound}

In this section, we prove Theorem \ref{thm:lower-bound}, showing that there exist twisted torus links with arbitrarily large volume.

\begin{proof}[Proof of Theorem \ref{thm:lower-bound}]
We will find a link $L_N$ in $S^3$ with volume at least $V+\e$, then
show that twisted torus knots are obtained by arbitrarily high Dehn
fillings of the components of the link.  By work of J{\o}rgensen and
Thurston, for high enough Dehn filling we will obtain a twisted torus
knot with volume at least $V$.

Consider again the Hopf link complement $S^3\setminus (C_1\cup C_2)$,
which is homeomorphic to $T^2 \times (-1,1)$.  Consider $T^2$ as the
unit square $[-1,1] \times [-1,1]$ with sides identified.  Let $L$ be
the link defined as the union of line segments $\{-1/2\} \times
\{1/2\} \times [-1/2,1/2]$, $\{1/2\}\times \{-1/2\} \times [-1/2,
  1/2]$, and lines from $(-1/2, 1/2)$ to $(1/2, -1/2)$ on the tori
$T^2 \times \{-1/2\}$ and $T^2 \times \{1/2\}$.

Chose $N \in \Z$ such that $v_3\,(2N+4) > V$, where $v_3$ is the
volume of a regular hyperbolic ideal tetrahedron.  Let $K_0$ be the
curve on $T^2 \times \{0\}$ with slope $0$, running through the center
$(0,0)$ of $T^2 = [-1,1]\times [-1,1] / \sim$.  Note $K_0$ does not
meet $L$.

Now, for $i=1, \dots, N$, if $i$ is odd, let $K_i$ and $K_{-i}$ be
curves of slope $1/0$ on $T^2 \times \{i/(2N+2)\}$ and $T^2 \times
\{-i/(2N+2)\}$, respectively.  If $i$ is even, let $K_i$ and $K_{-i}$
be curves of slope $0/1$ on $T^2 \times \{i/(2N+2)\}$ and $T^2 \times
\{-i/(2N+2)\}$, respectively.  Each $K_{\pm i}$ is required to be a
straight line on $T^2$, running through the center $(0,0)$ of $T^2 =
[-1,1]\times [-1,1] / \sim$.  An example for $N=2$ is shown in Figure
\ref{fig:twist-inside}\,(left).

Define the link $\widetilde L_N$ in $T^2\times (-1,1)$ by
$$\widetilde L_N = \left(L \cup K_0 \cup \left( \bigcup_{i=1}^N (K_i \cup
K_{-i}) \right) \right).$$
We identify $S^3\setminus (C_1\cup C_2)\cong T^2\times (-1,1)$, and
let $L_N$ denote $\widetilde L_N \cup C_1\cup C_2$ in $S^3$, so that
$L_N$ is a link in $S^3$ with $4+2N$ components.

\begin{lemma}
	$S^3\setminus L_N$ is hyperbolic, for any $N$.
\label{lemma:LN-hyp}
\end{lemma}

For readability, we will postpone the proof of Lemma
\ref{lemma:LN-hyp} until we have finished proving Theorem
\ref{thm:lower-bound}.  Assuming this lemma, since $S^3\setminus L_N$
is a hyperbolic manifold with $4+2N$ cusps, its volume is at least $(4+2N)\,v_3 > V$ by
work of Adams \cite{adams}.

For any positive integers $n_1, \dots, n_N$, perform Dehn filling
on $L_N$ as follows.  First, perform $1/n_1$ Dehn filling on $K_1$ and
$-1/n_1$ filling on $K_{-1}$.  The effect of this pair of Dehn
fillings is to twist along the annulus bounded by $K_1$ and $K_{-1}$.
See, for example, Baker \cite{baker} for an explanation of twisting
along an annulus.  Since $K_0$ is the only link component meeting this
annulus, this Dehn filling performs $n_1$ Dehn twists of $K_0$ about
the slope $1/0$ (corresponding to $K_1$ and $K_{-1}$), removes $K_1$
and $K_{-1}$, but otherwise leaves the link unchanged.

Now perform $1/n_2$ filling on $K_2$, and $-1/n_2$ filling on
$K_{-2}$.  Again the effect is a Dehn twist.  Continue for each $i$,
$i=1, \dots, N$.  The result is a manifold $M(p,q,r,0)$, where $p/q$ has continued fraction expansion  $[n_1, \ldots, n_N]$, and $r$ also
depends on $N$ and the integers $n_1, \dots, n_N$. See Figure
\ref{fig:twist-inside}.  
\begin{figure}
\includegraphics[height=1.2in]{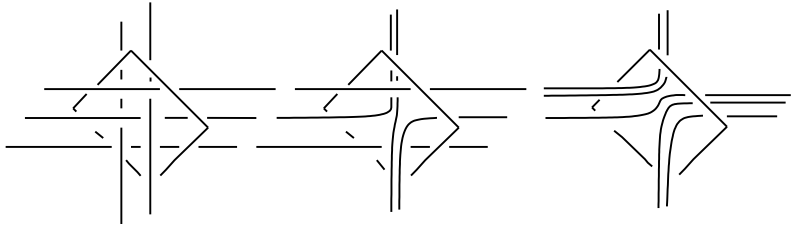}
\caption{Left:  Portion of $L_2$ shown.  Middle:  Perform $1/2$ and
	$-1/2$ Dehn filling on $K_1$ and $K_{-1}$, respectively.  Right:
	Then perform $1$ and $-1$ Dehn filling on $K_2$ and $K_{-2}$,
	respectively.}
\label{fig:twist-inside}
\end{figure}
By choosing $n_1, \dots, n_N$ to be large, we can ensure that $M(p,q,r,0)$ is hyperbolic, 
and its volume is arbitrarily close to that of $S^3\setminus
L_N$.

Now, obtain a twisted torus knot by performing $1/m$ Dehn filling on
$L$ in $M(p,q,r,0)$ and by performing Dehn filling on the Hopf link $C_1
\cup C_2$ in $M(p,q,r,0)$ along slopes with intersection number $1$.
Since there are infinitely many of these, we may choose these slopes
high enough that the result has volume arbitrarily close to that of
$S^3\setminus L_N$.  Since the volume of $S^3\setminus L_N$ is greater than $V$, this finishes the proof of Theorem \ref{thm:lower-bound}.
\end{proof}

\begin{proof}[Proof of Lemma \ref{lemma:LN-hyp}]
The proof is by induction.  One can check (by drawing the link explicitly and 
triangulating by hand or computer \cite{snappy})
 that $S^3\setminus L_1$ is hyperbolic. The manifold $S^3\setminus L_N$ is
obtained from $S^3\setminus L_{N-1}$ by removing the two closed curves
$K_N$ and $K_{-N}$.  We show that the manifold obtained by removing
$K_N$ from $S^3\setminus L_{N-1}$ is hyperbolic, and similarly the
manifold obtained by removing $K_{-N}$ from $S^3 \setminus (L_{N-1}
\cup K_N)$ is hyperbolic, assuming hyperbolicity of the previous
manifold.  The proofs for $K_N$ and for $K_{-N}$ are identical, so we
do them simultaneously.  To establish notation, call the initial
manifold $M_N$.  Assuming $M_N$ is hyperbolic, we show that
$M_N\setminus K$ is hyperbolic, where $K=K_N$ or $K=K_{-N}$.  Recall
that to show a link complement is hyperbolic, we need only show it is
irreducible, boundary irreducible, atoroidal and an-annular.

First, note that $K$ cannot be homotopically trivial in $M_N \subset
T^2\times I$, because it is parallel to the curve $1/0$ or $0/1$ on
$T^2$.

Now it follows from standard arguments that $M_N \setminus K$ is
irreducible and boundary irreducible, for any embedded essential
2--sphere in $M_N \setminus K$ would bound a ball in $M_N$, hence
contain $K$, which would mean $K$ is homotopically trivial in $M_N$,
contradicting the above paragraph.  Any boundary compression disk
would either have boundary on $K$, or would form half of an essential
2--sphere, in either case again implying $K$ is homotopically trivial.
So $M_N \setminus K$ is irreducible and boundary irreducible.

If $M_N \setminus K$ contains an essential annulus, a regular
neighborhood of the annulus and the link components on which its
boundary lies gives an embedded torus in $M_N \setminus K$.  If we can
prove $M_N \setminus K$ is atoroidal, then again standard arguments
will imply it is an-annular.

So it remains to show $M_N\setminus K$ is atoroidal.  Suppose otherwise: 
there exists an essential torus $T$ in $M_N \setminus K$.
Since $M_N$ is hyperbolic, $T$ is boundary--parallel or compressible in $M_N$. 
In either case, $T$ must bound a ``trivial'' $3$--dimensional submanifold $V \subset M_N$, where $V$ is one of the following:
\begin{enumerate}
\item $V = T^2 \times I$. This occurs when $T$ is boundary--parallel.
\item $V$ is a solid torus. 
\item $V$ is a ball-with-knotted-hole, contained in a ball in $M_N$.
\end{enumerate}

The last two cases occur if $T$ is compressible in $M_N$. In this case, surgering $T$ along its compressing disk $D$ will produce a sphere, which must bound a ball $B$ because $M_N$ is irreducible. If $B$ is disjoint from the compression disk $D$, then $V$ is obtained by adding a $1$--handle whose cross-section is $D$, and is a solid torus. If $B$ contains $D$, then $V$ is obtained by removing the $1$--handle whose cross-section is $D$, hence is a ball-with-knotted-hole.


%

\begin{figure}
\begin{overpic}[height=1.2in]{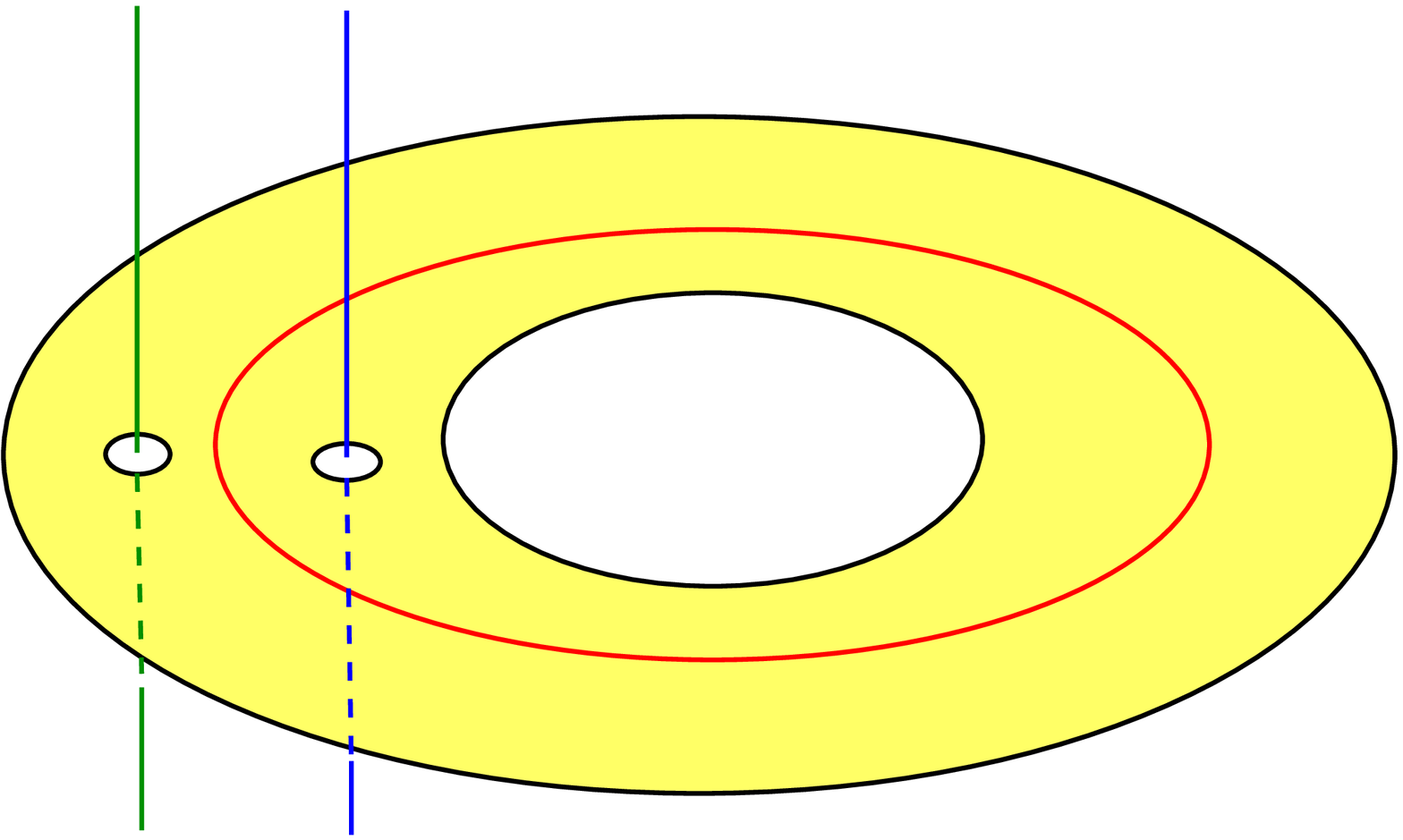}
\put(11,56){$L$}
\put(26,56){$K_{N-1}$}
\put(80,48){$\in T^2 \times \{ 1 \} $}
\put(87,28){$K_N$}
\put(53,28){$K_{N-2}$}
\end{overpic}
\caption{The knot $K = K_N$ lives on a four-punctured sphere in $M_N$.}
\label{fig:4punctsphere}
\end{figure}

It will help to consider the intersection between $T$ and the pair of pants $P$ whose three boundary components are a longitude of $K_{N-2}$, a meridian of $K_{N-1}$, and a longitude of $K_N = K$. ($P$ forms the inner half of the $4$--punctured sphere depicted in Figure  \ref{fig:4punctsphere}.) We assume that $P$ and $T$ have been moved by isotopy so as to minimize the number of curves of intersection. Because $P$ is incompressible, no curve of $P \cap T$ can be trivial on $T$. Thus all curves of intersection run in parallel along some non-trivial slope in $T$. Note that the pair of pants $P$ contains only $3$ isotopy classes of essential closed curve: these are parallel to the three boundary components. Because $P$ is essential, and the three boundary circles represent distinct elements of $\pi_1(M_N)$, the curves of $P \cap T$ that run in parallel on $T$ must also run in parallel along the same boundary component of $P$.

We consider the intersection $P \cap T$ to rule out the different types of trivial pieces enumerated above.

\begin{claim}\label{claim:not-parallel}
The torus $T$ cannot be boundary--parallel in $M_N$.
\end{claim}

\begin{proof}[Proof of claim]
Suppose that $T$ cuts off a product region $V = T^2 \times I$, adjacent to a boundary component of $M_N$. Note that, since $T$ is essential in $M_N \setminus K$, we must have $K \subset V$. Consider whether $K_{N-2}$ and $K_{N-1}$ also intersect $V$.

If both $K_{N-2}$ and $K_{N-1}$ are disjoint from $V$, then $T$ separates $K$ from  $K_{N-2}$ and $K_{N-1}$. But then the circles of $T \cap P$ run parallel to $K$, and $K$ is isotopic into $T$, hence into $\bdy M_N$. This contradicts the construction of $K$.

If $K_{N-2}$ intersects $V$, then $V$ is parallel to the boundary component of $M_N$ that corresponds to $K_{N-2}$. On the other hand, $K_{N-1}$ must lie outside $V$, because all of $\bdy V$ is already accounted for. Thus $T$ separates $K$ and $K_{N-2}$ from $K_{N-1}$, and the circles of $T \cap P$ run parallel to the meridian of $K_{N-1}$. But then there must be an essential annulus from the meridian of $K_{N-1}$ to the boundary torus corresponding to $K_{N-2}$. This contradicts the assumption that $M_N$ is hyperbolic.

If $K_{N-1}$ intersects $V$, then the argument is exactly the same, with $K_{N-1}$ and $K_{N-2}$ interchanged, and the longitude of $K_{N-2}$ in place of the meridian of $K_{N-1}$. Again, we get a contradiction.
\end{proof}

\begin{claim}\label{claim:not-solid}
The torus $T$ cannot bound a solid torus in $M_N$.
\end{claim}

\begin{proof}[Proof of claim]
Suppose that $T$ bounds a solid torus $V \subset M_N$. Then, because $T$ is incompressible in $M_N \setminus K$, we must have $K \subset V$. On the other hand, because all boundary components of $M_N$ are outside $V$, $K_{N-1}$ and $K_{N-2}$ must be outside $V$. Thus all circles of $T \cap P$ must be parallel to $K$. In particular, $K$ is parallel into the torus $T$.

Say that $K$ is an $(a,b)$ curve on $T$, which goes $a$ times around a meridian disk in $V$, and $b$ times around a longitude of $V$. Thus, in $\pi_1(T^2 \times (-1,1))$, $K$ represents $b$ times the generator of $\pi_1(V) = \mathbb{Z}$. But by definition, $K=K_N$ is a $0/1$ or $1/0$ curve on $T^2$, which is primitive in $\pi_1(T^2)$. Therefore, $b= \pm 1$, and the $(a, \pm 1)$ curve $K$ is isotopic to the core of $V$.

We conclude that $T$ is the boundary of a regular neighborhood of $K$, contradicting the assumption that it's essential in $M_N \setminus K$.
\end{proof}

\begin{claim}\label{claim:not-ball}
The torus $T$ cannot be contained in a ball in $M_N$.
\end{claim}

\begin{proof}[Proof of claim]
Suppose that the trivial piece $V$ bounded by $T$ is a ball-with-knotted-hole. Then $V$ is the complement of a tubular neighborhood of a knot $K' \subset S^3$. Note that $K'$ must truly be knotted, because by Claim \ref{claim:not-solid}, $V$ cannot be a solid torus.

Now, consider what happens to the pair $(M_N, K)$ when we Dehn fill all boundary components of $M_N$ along their meridians in $S^3$. Only the knot $K \subset S^3$ remains. By construction, $K$ is contained in $S^3 \setminus V$, which is a tubular neighborhood of $K'$. Furthermore, the torus $T$ is incompressible into $V$, and $K$ must intersect any compression disk of $T$ to the outside of $V$. Thus $T$ is incompressible in $S^3 \setminus K$, and $K$ is a satellite knot with companion $K'$. 

On the other hand, recall that $K = K_N$ is a curve of slope $0/1$ or $1/0$ on the torus $T^2$, hence parallel to one of the components of the Hopf link, and an unknot in $S^3$. This is a contradiction.
\end{proof}

By Claims \ref{claim:not-parallel}, \ref{claim:not-solid}, and \ref{claim:not-ball}, $T$ cannot be boundary--parallel or compressible in $M_N$. This violates the inductive hypothesis that $M_N$ is hyperbolic, and completes the proof that $M_N \setminus K$ is hyperbolic.
\end{proof}

\newpage
\bibliographystyle{hamsplain}
\bibliography{references}

\providecommand{\bysame}{\leavevmode\hbox to3em{\hrulefill}\thinspace}
\providecommand{\href}[2]{#2}
\begin{thebibliography}{10}

\bibitem{adams:3-punct}
Colin~C. Adams, \emph{Thrice-punctured spheres in hyperbolic {$3$}-manifolds},
  Trans. Amer. Math. Soc. \textbf{287} (1985), no.~2, 645--656.

\bibitem{adams}
\bysame, \emph{Volumes of {$N$}-cusped hyperbolic {$3$}-manifolds}, J. London
  Math. Soc. (2) \textbf{38} (1988), no.~3, 555--565.

\bibitem{baker}
Kenneth Baker, \emph{Surgery descriptions and volumes of {B}erge knots. {I}.
  {L}arge volume {B}erge knots}, J. Knot Theory Ramifications \textbf{17}
  (2008), no.~9, 1077--1097.

\bibitem{birman-kofman}
Joan Birman and Ilya Kofman, \emph{A new twist on {L}orenz links}, J. Topol.
  \textbf{2} (2009), no.~2, 227--248.

\bibitem{Birman-Gebhardt-Meneses}
Joan~S. Birman, Volker Gebhardt, and Juan Gonz{\'a}lez-Meneses,
  \emph{{Conjugacy in Garside groups III: Periodic braids}}, J. Algebra
  \textbf{316} (2007), no.~2, 746--776.

\bibitem{cdw:simplest}
Patrick~J. Callahan, John~C. Dean, and Jeffrey~R. Weeks, \emph{The simplest
  hyperbolic knots}, J. Knot Theory Ramifications \textbf{8} (1999), no.~3,
  279--297.

\bibitem{ckp:next-simplest}
Abhijit Champanerkar, Ilya Kofman, and Eric Patterson, \emph{The next simplest
  hyperbolic knots}, J. Knot Theory Ramifications \textbf{13} (2004), no.~7,
  965--987.

\bibitem{snappy}
Marc Culler, Nathan~M. Dunfield, and Jeffrey~R. Weeks, \emph{{SnapPy, a
  computer program for studying the geometry and topology of 3-manifolds}},
  {\tt http://\allowbreak snappy.\allowbreak computop.\allowbreak org}.

\bibitem{dean:sff}
John~C. Dean, \emph{Small {S}eifert-fibered {D}ehn surgery on hyperbolic
  knots}, Algebr. Geom. Topol. \textbf{3} (2003), 435--472.

\bibitem{fkp:dfvjp}
David Futer, Efstratia Kalfagianni, and Jessica~S. Purcell, \emph{Dehn filling,
  volume, and the {J}ones polynomial}, J. Differential Geom. \textbf{78}
  (2008), no.~3, 429--464.

\bibitem{fkp:symm}
\bysame, \emph{Symmetric links and {C}onway sums: volume and {J}ones
  polynomial}, Math. Res. Lett. \textbf{16} (2009), no.~2, 233--253.

\bibitem{fkp:farey}
\bysame, \emph{Cusp areas of {F}arey manifolds and applications to knot
  theory}, Int. Math. Res. Notices (2010), no.~23, 4434--4497,
  \mbox{arXiv:0808.2716}.

\bibitem{lackenby:volume-alt}
Marc Lackenby, \emph{The volume of hyperbolic alternating link complements},
  Proc. London Math. Soc. (3) \textbf{88} (2004), no.~1, 204--224, With an
  appendix by Ian Agol and Dylan Thurston.

\bibitem{purcell:hyp-genaug}
Jessica~S. Purcell, \emph{Hyperbolic geometry of multiply twisted knots}, Comm.
  Anal. Geom. \textbf{18} (2010), no.~1, 101--120, \mbox{arXiv:0709.2919}.

\bibitem{Rolfsen}
Dale Rolfsen, \emph{Knots and links}, Publish or Perish Inc., Berkeley, Calif.,
  1976, Mathematics Lecture Series, No. 7.

\bibitem{Thurston}
William~P. Thurston, \emph{{The Geometry and Topology of Three-Manifolds}}, \tt
  http://\allowbreak www.msri.org/\allowbreak publications/\allowbreak
  books/\allowbreak gt3m/.

\end{thebibliography}

\end{document}